\documentclass[11pt,letterpaper,twoside,english]{amsart}
\usepackage{hyperref}
\usepackage[foot]{amsaddr}
\usepackage{amsmath,amsthm,amssymb}
\usepackage{tikz,tikz-cd}
\usepackage{setspace}
\usepackage{comment}
\usepackage[backend=biber,style=alphabetic,maxnames=5]{biblatex}
\usepackage{graphicx}
\usepackage{xcolor}
\allowdisplaybreaks

\newtheorem{theorem}{Theorem}[section]
\newtheorem{prop}[theorem]{Proposition}
\newtheorem{lemma}[theorem]{Lemma}
\newtheorem{corollary}[theorem]{Corollary}
\theoremstyle{remark}
\newtheorem{definition}[theorem]{Definition}
\newtheorem{remark}[theorem]{Remark}

\def\bs{\backslash}

\def\bC{\mathbb{C}}
\def\bQ{\mathbb{Q}}
\def\bR{\mathbb{R}}
\def\bP{\mathbb{P}}
\def\bZ{\mathbb{Z}}
\def\bV{\mathbb{V}}

\def\cM{\mathcal{M}}
\def\cD{\mathcal{D}}

\def\tAut{\mathrm{Aut}}

\def\tdeg{\mathrm{deg}}

\def\tAut{\mathrm{Aut}}
\def\tSpec{\mathrm{Spec}}
\usepackage[mathscr]{eucal}

\def\olS{\overline S}

\usepackage{braket}
\usepackage{stmaryrd}
\usepackage{mathrsfs}
\usepackage{enumitem}

\numberwithin{equation}{section}

\numberwithin{equation}{section}

\begin{document}

\title{On the generic degree of two-parameter period mappings}
\author{Chongyao Chen and Haohua Deng}
\address{Department of Mathematics, Duke University, Durham, North Carolina, 27708-0320}
\email{cychen@math.duke.edu}
\email{haohua.deng@duke.edu}

\date{\today}

\maketitle

\small

\noindent\textbf{Abstract. }We present a method for computing the generic degree of a period map defined on a quasi-projective surface. As an application, we explicitly compute the generic degree of three period maps underlying families of Calabi-Yau 3-folds coming from toric hypersurfaces. As a consequence, we show that the generic Torelli theorem holds for these cases.
\normalsize
\section{Introduction}\label{sec01}

The period map is an important tool for studying families of algebraic varieties. Many nice properties associated to smooth projective varieties varying in a family are recorded by the associated polarized variation of Hodge structures (PVHS).

Torelli-type problems characterize the extent to which one could distinguish non-isomorphic projective varieties via their Hodge structures. There are different types of Torelli theorems. The global Torelli theorem is the strongest, stating that any two non-isomorphic projective varieties in a family can be completely distinguished by their polarized Hodge structures. There are classical examples for the global Torelli theorem like the universal family of principle polarized abelian varieties and moduli space of marked polarized K3 surfaces. One of the most recent examples is the family of mirror quintic Calabi-Yau threefolds studied in \cite{Fil23}. The infinitesimal Torelli theorem, on the other hand, states that the first-order deformation of a projective algebraic variety is mapped faithfully to its infinitesimal variation of Hodge structure. There are many more known examples for the infinitesimal Torelli theorem.

In this paper we consider the generic Torelli theorem, or more generally, the generic degree of a period map provided that the degree is well-defined. This is weaker than the global Torelli theorem in the sense that one has the global Torelli theorem only on a Zariski open subset, but it is still very useful. 

We introduce the first main result of this paper. Let $\Phi: S\rightarrow \Gamma\backslash D$ be a period map defined on a quasi-projective surface $S$. We also choose a projective completion $\olS$ of $S$ with $\olS-S$ being a simple normal crossing divisor. We show the following informally stated result:

\begin{theorem}\label{Thm:mainthmongenericdegree}
    The generic degree of $\Phi: S\rightarrow \Gamma\backslash D$ can be computed from the limiting mixed Hodge structure (LMHS) types and monodromy matrices around boundary points $s\in \olS-S$.
\end{theorem}

The main tool for our proof is Kato-Nakayama-Usui's theory on the space of nilpotent orbits \cite{KU08}, \cite{KNU10} which is based on the classical nilpotent and $\mathrm{SL}_2$-orbit theorems \cite{Sch73}, \cite{CKS86}. Since KNU's construction always exists for one-parameter period maps, several authors have successfully computed the generic degree of some one-parameter period maps as applications of the theory, see for example \cite{Usu08}, \cite{Shi09}, \cite{HK21}\footnote{Though the result in \cite{Usu08} is correct, it has one flawed argument which may cause problems in higher-dimensional cases. The arguments in this paper successfully fix it.}. 

Recently, C. Robles and the second author showed that KNU's construction exists for any two-parameter period maps \cite{DR23}. This allows us to prove Theorem \ref{Thm:mainthmongenericdegree} by only looking at some distinguished boundary points (and their image in the Kato--Nakayama--Usui space) like the one-parameter case, though the details are more complicated.

The second part of the paper is devoted to the study of specific examples. We study three different families of smooth Calabi-Yau threefolds with Hodge numbers $(1,2,2,1)$ that are hypersurfaces of complete toric varieties and compute the generic degree of each of these families using Theorem \ref{Thm:mainthmongenericdegree}. These families are the primary examples of study on the subject of mirror symmetry \cite{AGM94,HKTY95,CK99}. We denote them as the families $V_{(2,29)}$, $V_{(2,38)}$ and $V_{(2,86)}$ based on their Hodge diamonds. In particular, the family $V_{(2,86)}$ is also known as the family of mirror octics. 

For each of these Calabi-Yau hypersurfaces, we study the tautological family over the simplified moduli space. This family is a base change of the tautological family over complex moduli $\cM$ via a generically étale morphism $\phi:\cM_{\mathrm{simp}}\rightarrow \cM$. One advantage of this family is that $\cM_{\mathrm{simp}}$ admits a natural projective completion given by a complete toric surface. 
Now, we study the VHS associated to these families and calculate boundary LMHS types and monodromy matrices required in Theorem \ref{Thm:mainthmongenericdegree}.

First, the Picard-Fuchs system can be obtained from the GKZ system via Lemma \ref{lemquotmum}, which justifies the factorization argument first made in \cite{HKTY95} and rephrases it in the language of $\cD$-modules. Then one can calculate the discriminant locus from the Picard-Fuchs $\cD$-module. By blowing up the non-complete intersection points on the discriminant locus, we then obtain a smooth projective completion $\overline{\cM}_{\mathrm{simp}}$ of $\cM_{\mathrm{simp}}$, such that $\overline{\cM}_{\mathrm{simp}}- \cM_{\mathrm{simp}}$ is a simple normal crossing divisor as desired. Next, one can obtain the Gauss-Manin connection represented in a chosen local basis from the Picard-Fuchs $
\cD$-module. Then the log-monodromies represented in this basis are just the (Tate twisted) residues of the Gauss-Manin connection in suitable local coordinates. Meanwhile, the symplectic form can be obtained by calculating the intersection matrix of the local basis, which can be written in terms of the Yukawa couplings and $n$-point functions.
Finally, the types of the LMHS can be obtained via the Jordan normal form of the nilpotent operators and the polarized period relations \cite{KPR19}, as the first author did in \cite{Chen23}.

As a consequence of computations and application of Theorem \ref{Thm:mainthmongenericdegree}, we arrive at the second main theorem of this paper:
\begin{theorem}\label{Thm:mainthmcomputedegree}
    The period maps associated to the families $V_{(2,29)}$, $V_{(2,38)}$, and $V_{(2,86)}$ (mirror octic) over simplified moduli have generic degree $2$, $1$, and $1$ respectively.
\end{theorem}

Recall that a smooth variety is said to satisfy the generic Torelli theorem if the period map associated to the tautological family over complex moduli has generic degree $1$. Theorem \ref{Thm:mainthmcomputedegree} and the computation of the generic degrees of $\phi:\cM_{\mathrm{simp}}\rightarrow \cM$ in Section \ref{sec04} also imply:
\begin{corollary}
    The Calabi-Yau threefolds $V_{(2,29)}$,  $V_{(2,38)}$, and $V_{(2,86)}$ satisfy the generic Torelli theorem.
\end{corollary}
To the best of our knowledge, these are the first known examples of Calabi-Yau threefolds that satisfy the generic Torelli theorem with complex moduli larger than one dimension.

\medskip

The article is organized as follows. In Section \ref{sec02} we summarize all necessary materials in toric geometry and Hodge theory. In Section \ref{sec03} we present and prove the main formula for computing generic degree. Section \ref{sec04} is devoted to the study of the three families $V_{(2,29)}, V_{(2,38)}$ and $V_{(2,86)}$, including explicitly their boundary LMHS types and monodromy matrices. In Section \ref{sec05} we will apply the main theorem proved in Section \ref{sec03} and computational results from Section \ref{sec04} to show Theorem \ref{Thm:mainthmcomputedegree}.

\medskip
\noindent\textbf{Acknowledgement:} The authors thank Colleen Robles for exchanges of many insightful ideas. The second author also thanks Matt Kerr for related discussions.
\section{Background materials}\label{sec02}

\subsection{Toric geometry}
In this section, we fix the notations that will be used in the next section. These are standard and follow closely those in \cite{CK99,CLS11}.

\subsubsection{Toric varieties}
Denote the standard full-rank integral lattice group in $\bQ^n$ as $N$ and its dual lattice as $M$. 
We will use the $k\times k$ matrices $\frac{A_{n\times k}}{B_{(k-n)\times k}}$ to record the toric data. The column vectors of $A$ form a finite set 
\[
\{v_1,v_2,\dots,v_k\}=:\Xi\subset N,
\]
and the set of row vectors $\{w_1,w_2,\dots,w_{r}\}$ ($r := k-n$) of $B$ generates the lattice of relations among $\Xi$
\[
\Lambda:=\left\{l=(l_i)\in \bZ^k: \sum_{i=1}^k l_iv_i=0\right\}
\]
over $\bZ$. 
In particular, we have $A\cdot B^t=0$. These data can be viewed as recording either the information about a polyhedral fan $\Sigma$ or the information of an integral polytope $\Delta^\circ$. Each of these points of view leads to one way to construct a toric variety, which we now briefly review.

For a strongly convex rational polyhedral fan $\Sigma\subset N_\bR:=N\otimes_\bZ\bR$, denote $\Sigma(m)$ as the set of its $m$-dimension cones. In this case, we use the columns of $A$ to record the primitive generators of elements in $\Sigma(1)$. Denote $\sigma_{i_1,\dots,i_m}\in \Sigma(m)$ as the $m$-dimensional cone generated by $\{v_{i_1},v_{i_2},\dots,v_{i_m}\}$. 
The toric variety constructed from $\Sigma$ is denoted as $X_\Sigma$ and is called the \textit{GIT quotient construction}. More precisely, let $S = \bC[x_1,\dots,x_k]$.
\begin{definition}The \textit{Cox ideal} $I_{\Sigma}$ is the ideal of $S$ generated by the monomials corresponding to each $\sigma\in\Sigma(n)$, defined by
\[
x_{i_1}x_{i_2}\dots x_{i_\alpha},\quad  \alpha = k-\sigma(1),\,\bR_+\cdot v_{i_s}\notin\sigma(1).
\]
\end{definition}
\noindent In particular, we have 
\begin{align*}
X_\Sigma  & = \mathrm{Proj}_\Sigma(S)\\
& = [\mathrm{Spec(S)- \bV(I_\Sigma)}]\sslash G,
\end{align*}
where the $G:=(\bC^*)^r$-action on $\bC^k$ is determined by the matrix $B$ as
\begin{align*}
G\times \bC^k & \rightarrow  \bC^k\\
(s_1,\dots,s_r)\times(x_1,\dots,x_k) & \rightarrow \left(\prod_{i=1}^rs_i^{w_{i,1}}\cdot x_1,\cdots,\prod_{i=1}^rs_i^{w_{i,k}}\cdot x_k\right)
\end{align*}
This construction suggests that one could think about the toric variety as a generalization of weighted projected space. More precisely, the degree of a variable is replaced by a multi-degree, and the origin that got deleted before the quotient is replaced by $\bV(I_\Sigma)$. For the variable $x_i$, the multi-degree is $\tdeg(x_i) = (w_{1,i},\dots,w_{r,i})$. We note this multi-degree can be viewed as taking value in $A_{n-1}(X_\Sigma)\otimes \bZ$.

For an integral polytope $\Delta^\circ$ i.e., a polytope in $\bR^k$ whose vertices lie in $N$. The toric data records all the integral points inside the polytope $\Delta^\circ$, i.e., $\Xi=\Delta^\circ \cap N$. We will use $\Delta^\circ_{i_1,\dots,i_m}$ to denote the $m-1$-dimensional face of $\Delta^\circ$ that has $v_{i_1},\dots,v_{i_m}$ as its vertices. Denote the toric variety constructed from it as $\bP_{\Delta^\circ}$, which is the Zariski closure of the open immersion
\begin{equation}\label{Taction}
    \begin{split}
    (\bC^*)^n &\hookrightarrow \bP^k\\
    (t_1,\dots,t_n)& \rightarrow \left[\prod_{i=1}^nt_i^{r_{1,i}}:\prod_{i=1}^nt_i^{r_{2,i}}:\cdots:\prod_{i=1}^nt_i^{r_{k,i}}\right].    
    \end{split}
\end{equation}
This also induces a $(\bC^*)^n$-automorphism on $\bP_{\Delta^\circ}$. We will call this construction the $\textit{projective embedding construction}$.

Each of the two constructions contains a top dimensional Zariski dense tori $T:=(\bC^*)^n$ inside the toric variety. For $\bP_{\Delta^\circ}$ it is clear from the construction, whereas for $X_\Sigma$, we have $T = (\bC^*)^k/G\hookrightarrow X_{\Sigma}$.

\subsubsection{Reflexive polytope and Baytrev mirror}
The starting point of a Baytrev's mirror construction is a reflexive polytope and its mirror, we first recall some definitions.
\begin{definition}\label{defnf}
For an arbitrary full-dimensional integral polytope $\Delta\subset M_\bR$. For each codimension $s$ face $F$ of $\Delta$, let $\sigma_F$  be the $s$-dimensional strongly convex rational polyhedral cone in $N_\bR$ defined as
\[
\sigma_F:=\{v\in N_\bR|\braket{m,v}\geq\braket{m',v} ,\forall m\in F,m'\in \Delta\}.
\]
Then the \textit{normal fan} $\Sigma_\Delta$ of $\Delta$ is the union of all $\sigma_F$.      
\end{definition}

\begin{definition}\label{defref}
Let $\Delta$ be the same in Definition \ref{defnf}, then $\Delta\subset M_\bR$ is called \textit{reflexive} if for any facet (codimension-1 face) $\Gamma$, there is a unique normal vector $v_\Gamma\in N$, such that
$$
\Delta = \{m\in M_\bR|\braket{m,v_\Gamma}\geq-1,\forall \ \Gamma\leq \Delta\}.
$$
and $\mathrm{Int}(\Delta) \cap N = \{0\}$.
The \textit{polar dual} $\Delta^\circ\subset N_\bR$ of $\Delta$ is defined as
$$
\Delta^\circ:=\{n\in N_\bR|\braket{m,n}\geq -1, \forall m\in \Delta\}.
$$
\end{definition}
If $\Delta$ is reflexive, $\Delta^\circ$ is also reflexive and we have $(\Delta^\circ)^\circ = \Delta$. Furthermore, $\Delta^\circ$ is the convex hull of $v_\Gamma$, and $\{v_\Gamma\}$ spans $\Sigma_\Delta(1)$. For each face $F$ of $\Delta$, one can define its \textit{dual face} to be
\[
F^\circ:=\bigcap_{F\subset\Gamma} \Gamma \subset \Delta^\circ,
\]
where $\Gamma$ runs through all the facets of $\Delta$ that contains $F$. One can easily show that $(F^\circ)^\circ = F$ and $\dim(F)+\dim(F^\circ) = \dim \Delta-1$.

These constructions also make the comparison of the two different constructions of the toric variety possible. In fact, we have $\bP_{\Delta^\circ}\cong X_{\Sigma_{\Delta}}$, and $\bP_{\Delta}\cong X_{\Sigma^\circ}$, where $\Sigma^\circ: = \Sigma_{\Delta^\circ}$.

By Baytrev mirror construction, the mirror of $X_\Sigma$, denote as $\check X_{\Sigma}$ is $\check X_{\Sigma}\cong X_{\Sigma^\circ}\cong \bP_{\Delta^\circ}$. On the level of the Calabi-Yau hypersurface, we have $\check X$ as the anti-canonical hypersurface of $\check X_{\Sigma}$ that is mirror to $X$.

\subsubsection{Automorphisms of $X_\Sigma$}
The automorphism group $\tAut(X_\Sigma)$ of $X_\Sigma$ is generated by three types of automorphisms that come from the $T$-action induced from \eqref{Taction}, the root symmetry, and the fan symmetry. The first two types of automorphism are continuous, together they generate the connected component $\tAut_0(X_\Sigma)$ of $\tAut(X_\Sigma)$.

The root symmetry is defined as follows, for each of the variables $x_i,i=1,\dots,k$ in the GIT quotient construction, if there exists a monomial $x^{\alpha}$ in $S$, such that $x_i\nmid x^\alpha$, and $\tdeg(x_i) = \tdeg(x^\alpha)$. Then the \textit{root pair} $(x_i,x^{\alpha})$ defines an $\bC$-automorphism on $X_\Sigma$
\[
\begin{split}
    \bC\times X_\Sigma & \rightarrow X_\Sigma \\
    (\lambda,[x_1:\dots:x_k]) & \rightarrow [x_1:\dots:x_{i-1}:x_i+\lambda x^\alpha:x_{i+1}:\dots:x_k].
\end{split}
\]
Therefore, the group of root symmetry $\tAut_r(X_\Sigma)$ is the group generated by all the possible root pairs.

The last type of automorphism comes from the symmetry of the $\Sigma$. More precisely, these automorphisms form a finite group $\tAut(\Sigma)$, the subgroup of $\tAut(N)$ that preserves $\Sigma$.

\subsubsection{Moduli spaces}
There are three moduli spaces that are naturally associated to $V$, i.e., the polynomial moduli, the simplified moduli, and the complex moduli. The polynomial moduli is
\[
\cM_{\mathrm{poly}} := \bP(L(\Delta^\circ\cap N))^{s}/\mathrm{Aut}(X_{\Sigma^\circ}),
\]
where $L(\Delta^\circ\cap N)$ is the vector space of Laurent polynomials associated to the finite set $\Delta^\circ\cap N$. Meanwhile, the superscript $s$ refers to the restriction to the smooth locus, which is quasi-projective.
This space parameterized the complex structure of $V$ that can be realized as hypersurfaces of $X_\Sigma$. The simplified is
\[
\cM_{\mathrm{simp}}:= \bP(L((\Delta^\circ\cap N))_0)^s/T,
\]
where $(\Delta^\circ\cap N))_0$ is the subset of $\Delta^\circ\cap N$ that excludes the points that lies in the interior of any facet of $\Delta^\circ$.

The \textit{dominance theorem} conjectured in \cite{AGM93} and proved in \cite{CK99} states that 
\begin{theorem}
    The quotient map $\phi:\cM_{\mathrm{simp}}\rightarrow \cM_{\mathrm{poly}}$ is generically {\'e}tale.
\end{theorem}
So if $\tAut_r(X_\Sigma)=0$, then the generic degree of $\phi$ is the order of the subgroup of $\tAut(\Sigma)$ that acts on $\cM_{\mathrm{simp}}$ non-trivially. Finally, the complex moduli $\cM$ is the moduli space of the complex structure. $\cM$ and $\cM_{\mathrm{poly}}$ are related by the following proposition
\begin{prop}
    The following are equivalent:
    \begin{enumerate}
        \item $\cM = \cM_{poly}$.
        \item $\dim \cM = \dim \cM_{poly}$.
        \item For any two-dimensional face of $\Delta$, either it has no interior point, or its dual face in $\Delta^\circ$ has no interior point.
    \end{enumerate}
\end{prop}
\begin{proof}
    See Proposition 6.1.3 of \cite{CK99}.
\end{proof}

\subsubsection{GKZ system}
A \textit{GKZ $A$-hypergeometric system} $\tau(A,\beta)$ is a $\cD_{\bC^{n}}$-module that is determined by a tuple $(A,\beta)$, where $A\in M_{m\times n}(\bZ)$ and $\beta\in \bZ^{m}$. For an introduction to the theory of $\cD$-module, we refer to \cite{HTT07}.

For any Calabi-Yau hypersurface inside a toric variety whose toric data is recorded in $(\frac{A}{B})$, one can associate it with a GKZ system $\tau(\bar A,\beta)$ that is constructed as follows. First we define the \textit{suspended fan} $\bar\Sigma$ of $\Sigma$, whose 1-dimensional cones are generated by $\bar \Xi :=(\Xi\cup\{0\})\times 1$. The toric data of the suspended fan will be recorded as $(\frac{\bar A}{\bar B})$. Then $\tau(\bar A,\beta)$ is the cyclic $\mathcal D_{\bC^{k+1}}$-module, determined by the left ideal $I_{GKZ}$ in Weyl algebra, which is generated by
\[
\prod_{l_i>0}\partial_{\lambda_i}^{l_i}-\prod_{l_j<0}\partial_{\lambda_j}^{l_j},\quad l\in\Lambda,
\]
and 
\[
\sum_{j=1}^{k+1}r_{i,j}\lambda_j\partial_j+\beta_i,\quad i=1,\dots,n+1,
\]
where $\lambda_i$ are the coordinates of $\bC^{k+1}$, and $\beta = (0,\dots,0,-1)$. Now, recall the general result of the GKZ system
\begin{theorem}[Hotta]
    Let $A\in M_{m\times n}(\bZ)$, if $(1,1,\dots,1)$ lies in the row span of $A$, then $\tau(A,\beta)$ is regular holonomic for any $\beta\in \bC^m$.
\end{theorem}
\begin{proof}
    See \cite{Hot91}.
\end{proof}
Then the GKZ system we constructed above is always regular holonomic. Now consider, Let $j:(\bC^*)^{k+1}\hookrightarrow \bC^{k+1}$ be the immersion, and $\pi:(\bC^*)^{k+1}\rightarrow (\bC^*)^{r}$ be the projection under the action \eqref{Taction}.
Then $p_+j^*\tau(\bar A,\beta)$ is again regular holonomic by the general theory of $\cD$-modules, (see e.g. \cite{HTT07}). 
\begin{lemma}\label{GKZsys}
    $p_+j^*\tau(\bar A,\beta)$ is a regular holonomic and is determined by the left ideal sheaf $\mathcal I_{GKZ}\subset \cD_{(\bC^*)^r}$ that is locally generated by the following $r$ differential operators ($ i = 1,\dots,r$)
    \[
    P_i:=\prod_{j:w_{i,j}>0}\prod_{k=1}^{w_{i,j}
    }(\sum_{l=1}^rw_{l,j}\delta_l-k-\delta_{j,k+1})-z_i\prod_{j:w_{i,j}<0}\prod_{k=1}^{w_{i,j}
    }(\sum_{l=1}^rw_{l,j}\delta_l-k-\delta_{j,k+1}),
    \]
    where $\delta_l := z_l\partial_{z_l}$ with $z_l$, $l=1,\dots,r$ be the coordinates on $(\bC^*)^r$, and $\delta_{i,j}$ is the Kronecker symbol.
\end{lemma}
\begin{proof}
   By \cite{GKZ}, the singular locus (the principal $\bar A$-deteminant) of $j^*\tau(\bar A,\beta)$ is away from the origin, so locally near origin, we have $\pi$ is smooth. Then by \cite{HTT07}, the push forward along a smooth morphism is just a change of variables.  For more detail see \cite{GKZ,CK99}. 
\end{proof}

\begin{remark}
The GKZ system can be viewed as a special case of the tautological system \cite{LSY13}. In that regard, the suspension corresponds to including the Euler operator in the system and is the only part specialized to the anti-canonical section.   
\end{remark}
In what follows, we will refer to the $\cD$-module $p_+j^*\tau(\bar A,\beta)$ as the GKZ system, and denote it as $\tau_{GKZ}$.

The singular locus of the GKZ system $\tau_{GKZ}$ is the so-called \textit{principal $A$-determinant} and it has a factorization as
\[
E_A = \prod_{F\subset \Delta} D_{(\Delta^\circ\cap N)_0\cap F}^{m(F)},
\]
where $F$ runs through all the faces of $\Delta$, and $D_{(\Delta^\circ\cap N)_0\cap F}$ are the so-call $A$-discriminants. We refer to  \cite{GKZ,CK99} for the precise definition of the right-hand side. In practice, it is often easier to calculate the singular locus directly, which equals the projection of the characteristic variety minus the zero section.

\subsubsection{Secondary fan}
The space $\cM_{\mathrm{simp}}$ has a natural compactification given by the \textit{Chow quotient}\cite{KSZ91}
\[
\overline{\cM}_{\mathrm{simp}} := \bP(L((\Delta^\circ\cap N))_0)\sslash_c T.
\]
This is an $r$-dimensional toric variety whose toric data is denoted as $(\frac{A^s}{B^s})$, where the columns of $A^s$ are the distinct primitive vectors of the columns of $\overline B$. The fan $\Sigma^s$ of $\overline{\cM}_{\mathrm{simp}}$ is called the \textit{secondary fan}, and can be determined by calculating the \textit{GKZ decomposition}, see for example \cite{CK99}. In the case $r\leq 2$, the $\Sigma^s$ is completely determined by $\Sigma^s(1)$.

The GKZ system $\tau_{GKZ}$ thus can be viewed as sitting on the dense tori $(\bC^*)^r$ of $\overline{\cM}_{\mathrm{simp}}$. Or one can consider the \textit{minimal extension} $L(\tau_{GKZ},(\bC^*)^r)$ to $\overline{\cM}_{\mathrm{simp}}$. This is a holonomic $\cD_{\overline{\cM}_{\mathrm{simp}}}$-module whose singular locus is contains inside
\[
\mathrm{Sing}(\tau_{GKZ}) = E_A\cup D_{v_1^s}\cup \dots \cup D_{v_q^s},
\]
where $v_i^s$, $i=1,\dots,q$ are the volumes of $A^s$, and $D_{v_i^s}$ is the toric divisor of $\overline{\cM}_{\mathrm{simp}}$ associated to $v_i^s$.

The geometry of $\overline{\cM}_{\mathrm{simp}}- \cM_{\mathrm{simp}}$ is well-understood. In \cite{GKZ} the authors show $E_A$ intersection 
of each $D_{v_i^s}$ at exactly one point with possible multiplicity. Meanwhile, the intersections between the toric divisors are recorded in the secondary fan. Therefore, by blowing up the tangencies at $E_A\cap D_{v_i^s}$, one arrives at a compactification of $\cM_{\mathrm{simp}}$ with normal crossing boundary divisors.

\subsubsection{Picard-Fuchs system}
Let $\bP_{\Delta^\circ}$ be a toric variety with at most terminal singularity, and let $V\subset \bP_{\Delta^\circ}$ be a smooth anti-canonical hypersurface. We are interested in the tautological family $\mathcal C\rightarrow \cM$. The Picard-Fuchs system $\tau_{PF}(\omega)$ for $V$ is a $\cD_{\cM}$-module. At a generic non-singular point $b\in \cM$, the local holomorphic solution germ of the Picard-Fuchs system is a finite dimension $\bC$-vector space that is spanned by the period functions of a fixed local section $\omega$ of the relative canonical sheaf
$K_{\mathcal C/\cM}$.

If we further assume $\cM_{\mathrm{poly}}=\cM$, then $\cM_{\mathrm{simp}}$ is a (possibly ramified) finite cover of $\cM$. In this case, $\cM_{\mathrm{simp}}$ parameterized the isomorphism class of $V$ together with an isomorphism of $\bP_{\Delta^\circ}$. Thus $\cM_{\mathrm{simp}}$ still carries a tautological family $\mathcal C'\rightarrow \cM_{simp}$, and one can also construct a Picard-Fuchs system by fixing a local section of $K_{\mathcal C'/\cM_{\mathrm{simp}}}$. In \cite{BC94} a canonical invariant section of $K_{\mathcal C'/\cM}$ is defined as
\[
\omega = \mathrm{Res}_{f=0}\frac{\Omega}{f},
\]
where $\Omega$ is a $T$-invariant meromorphic holomorphic $m$-form on $\bP_{\Delta^\circ}$ that does not depend on the base point.

The construction of the GKZ system implies \cite{HKTY95,LSY13} that on $\cM_{\mathrm{simp}}$, $\tau_{PF}$ is a quotient module of $\tau_{GKZ}$. Since both of these $\cM_{\mathrm{simp}}$-modules are cyclic with generator $\omega$, then they have a canonical filtration
\[
F_\bullet\tau_{PF,GKZ}:=F_\bullet \cD_{\cM_{\mathrm{simp}}} \cdot \omega.
\]
In particular, they both defines a $\bC$-VHS on the non-singular locus, with $(\tau_{PF},F_\bullet\tau_{PF})\cong \mathcal (\bV,F^{m-1-\bullet} \bV)$, where $(\bV,F^\bullet \bV)$ is the $\bC$-VHS associated to the tautological family.

Now, for the case $m=4$, i.e., $V$ is a Calabi-Yau 3-fold, we have
\begin{lemma}\label{lemquotmum}
    If $\tau_{GKZ}$ has holonomic rank less than $4r+4$, and it has a quotient module $\cM$ with holonomic tank $2r+2$, such that $(\cM,F^\bullet\cM) $ has a MUM point, where $F^\bullet \cM$ is the induced filtration from the canonical filtration, then generically $\tau_{PF} \cong \cM$.
\end{lemma}
\begin{proof}
    For the definition of MUM point see Definition \ref{defmum}. As pointed out in \cite{Chen23}, the existence of a MUM point implies the VHS is irreducible. Then since $\tau_{PF}$ is a quotient module of $\tau_{GKZ}$ with holonomic rank $2r+2$, thus generically it has to isomorphic to $\cM$.
\end{proof}

We also note that in \cite{HLY16}, the authors showed that for a smooth Calabi-Yau $(m-1)$-fold $V\subset \bP_{\Delta}$, the holonomic rank of $\tau_{GKZ}$ is equals to $\dim(H^m(\bP_{\Delta^\circ}- V,\bC))$, which is greater than $\dim(H^{m-1}(V,\bC))$.

\subsection{Hodge theory}

Throughout this subsection we fix the tuple 
\begin{center}
$(H_\bZ, Q, \{h^{p,q}\}_{p+q=l}, D)$,
\end{center}
where $(H_\bZ, Q)$ is an integral polarized lattice and $\{h^{p,q}\}_{p+q=l}$ is a set of Hodge numbers for some polarized Hodge structure of weight $l$ on $(H_\bZ, Q)$ and $D$ is the corresponding period domain. We also denote $D=G_\bR/K\subset \check{D}=G_\bC/P$ where $G:=\mathrm{Aut}(H, Q)$.

Let $S$ be a smooth quasi-projective variety, $\bar{S}$ be projective with $\bar{S}\backslash S$ a simple normal crossing divisor. We assume there is a $\mathbb{Z}$-local system $\bV$ on $S$ which induces a polarized variation of Hodge structures (PVHS) of given type. This induces a period map:
\begin{equation}\label{periodmapgeneralform}
    \varphi: S\rightarrow \Gamma \backslash D
\end{equation}
 where $D$ is the classifying space of $\bZ$-polarized Hodge structures of type $(H_\bZ, Q, \{h^{p,q}\}_{p+q=l})$, and $\Gamma\leq G=\mathrm{Aut}(H_\bZ, Q)$ is the monodromy group. 

 These abstract settings have natural models in algebraic geometry: Suppose 
 \begin{equation}
     \pi: \mathcal{X}\rightarrow S
 \end{equation}
is a smooth projective family where $\mathcal{X}$ is smooth and $\pi$ is a holomorphic proper submersion. Moreover suppose for any $s\in S$, the fiber $X_s:=\pi^{-1}(s)$ is a smooth projective variety of fixed dimension $\mathrm{dim}_{\mathbb{C}}X_s=l$. Denote $\mathbb{Z}_{\mathcal{X}}$ as the $\mathbb{Z}$-constant sheaf on $\mathcal{X}$, then 
\begin{equation}
    \mathbb{V}:=R^l\pi_*(\mathbb{Z}_{\mathcal{X}})
\end{equation}
is a $\mathbb{Z}$-local system over $S$ whose fiber at $s\in S$ is $H:=H^l(X_s, \mathbb{Z})/\mathrm{torsion}$. Take the Hodge decomposition 
\begin{equation}
    H_{\mathbb{C}}=H:=(H^l(X_s, \mathbb{Z})/\mathrm{torsion})\otimes \mathbb{C}=\oplus_{p+q=l}H^{p,q}(X_s)
\end{equation}
on each fiber into consideration, this gives a variation of Hodge structure of weight $l$ on $S$ as well as a period map of the form \eqref{periodmapgeneralform}.

For any $s\in \olS$, a local neighborhood around $s\in U\subset \olS$ satisfies $U\cap S\cong (\Delta^*)^k\times \Delta^{n-k}$ on which we may consider the local monodromy operators $\{T_i\}_{1\leq i\leq k}$ and their logarithms. One of the main ingredients is the following intepretation of Schmid's nilpotent orbit theorem:
\begin{theorem}\label{Thm:SchmidNilp}
    For every $s\in \olS$, up to the action of $\Gamma$ there is a nilpotent orbit $(\sigma_s, F_s)$ canonically associated to $s$. Here $\sigma_s$ is the local monodromy nilpotent cone and $F_s\in \check{D}$.  
\end{theorem}
\noindent There are two ways to intepretate nilpotent orbit $(\sigma_s, F_s)$:
\begin{enumerate}
    \item Suppose $\sigma_s=\langle N_1,...,N_r\rangle$, for local coordinates $\{t_i\}$ around $s$, the map $\psi(t_i)=\exp(\sum -\frac{\log(t_i)}{2\pi i}N_i)F_s$ approximates the local lift of $\varphi$ at $s$ (\cite[Sec. 4]{Sch73}).
    \item By \cite{CK82}, for every $s\in \olS$ there is a monodromy weight filtration $W_s:=W(\sigma_s)$ associated to $s$, such that $(W_s, F_s, \sigma_s)$ is a limiting mixed Hodge structure (LMHS).
\end{enumerate}
We give a better interpretation of (1) above. Suppose $U\subset \olS$ is an open subset such that $U\cap S\cong (\Delta^{*})^k\times\Delta^l$. Let $\mathfrak{H}$ be the upper-half plane $\{\mathfrak{Im}(z)>0\}$. The local period map $\Phi$ on $U\cap S$ has the local lift:
\begin{equation}
\begin{tikzcd}
\mathfrak{H}^k\times \Delta^l\arrow[d, ""] \arrow[r, "\tilde{\Phi}"] & D \arrow[d] \\
(\Delta^{*})^k\times\Delta^l \arrow[r, "\Phi"] & \Gamma \backslash D
\end{tikzcd}
\end{equation}
If we denote $z=(z_k), w=(w_l)$ as coordinates on $\mathfrak{H}^k$, $\Delta^l$,
\begin{equation}\label{eqn:localliftperiodmap}
    \tilde{\Phi}=\exp(\sum_{1\leq j\leq k}z_jN_j)\psi(e^{2\pi iz}, w),
\end{equation}
where $\psi(z,w)\in D$ is holomorphic over $\Delta^{k+l}$. Schmid's nilpotent orbit theorem \cite[Thm. 4.12]{Sch73} also says:
\begin{theorem}\label{Thm:SchmidNilp2}
    For any $w\in\Delta^l$, $(\sigma:=\langle N_1,...,N_k\rangle, \psi(0, w))$ is a nilpotent orbit, and under the canonical metric $d$ on $D$, as $\mathfrak{Im}(z)\rightarrow \infty$,
    \begin{equation}
        d(\exp(\sum_{1\leq j\leq k}z_jN_j)\psi(0, w), \tilde{\Phi}(e^{2\pi iz},w)) \sim e^{-b\mathfrak{Im}(z)}
    \end{equation}
    for some $b>0$.
\end{theorem}
For a general mixed Hodge structure $(W, F)$ on $(H_\bZ, Q)$, there is a canonical splitting by Deligne:
\begin{equation}\label{eqn:delignesplitting}
    H_\bC=\oplus_{p,q}H^{p,q}
\end{equation}
such that $W_lH_\bC=\oplus_{p+q\leq l}H^{p,q}$ and $F^lH_\bC=\oplus_{p\geq l}H^{p,q}$. We have the following definitions:
\begin{definition}\label{Def:MHSType}
   The natural numbers $\{h^{p,q}:=\mathrm{dim}H^{p,q}\}$ are called the Hodge numbers of the mixed Hodge structure. We say two mixed Hodge structures on $(H_\bZ, Q)$ have the same type if they have the same Hodge numbers. In particular, we say a mixed Hodge structure $(W,F)$ has the type given by $\{h^{p,q}\}$ if $\{h^{p,q}\}$ are Hodge numbers of it.
\end{definition}
\begin{remark}
    When $\sigma=\langle N_1, N_2\rangle$ and $F\in \check{D}$ such that $(\sigma, F)$ is a nilpotent orbit, we say $(\sigma, F)$ is of type $\langle A|B|C\rangle$ if $(N_1, F), (\sigma, F), (N_2, F)$ are of type $A,B,C$ correspondingly.
\end{remark}
Therefore, the period map \eqref{periodmapgeneralform} associates each $s\in \olS$ a $\exp(\sigma_{s,\bC})\Gamma$-class of LMHS $(W_s, F_s, \sigma_s)$ which gives a type of LMHS. It is known that the type of LMHS at all $s\in \olS$ is constant along each stratum.

We refer to \cite{KPR19} for classifying all possible Hodge diamonds of LMHS. For the sake of convenience, we list all possible LMHS types for the weight $3$ period domain of Hodge type $(1,r,r,1)$, $r\geq 1$ in terms of Hodge-Deligne diagrams (See also \cite[Example 5.8]{KPR19}). In particular, the type $\mathrm{IV}_r$ LMHS is also called the Hodge-Tate type degeneration. The following definition from \cite{CK99} will be useful in the rest of this paper.
\begin{definition}[MUM point]\label{defmum}
    Let $s\in \overline S-S$. Then $s$ is a maximal unipotent monodromy (MUM) point if the following is true
    \begin{enumerate}
        \item $s \in D_1\cap \dots \cap D_r$, where $D_i\subset \overline S -S$ are distinct irreducible divisors of $\overline S$.
        \item  Denote the monodromy operator associated to $D_i$ as $T_i$. Then $T_i$ are unipotent. 
        \item Denote $N_i:=\log(T_i)$, and the cone formed by them as $\sigma_s$. The LMHS $(W_s,F_s,\sigma_s)$ associated to $s$ is Hodge-Tate.
        \item Let $e_0$ be the generator of $(W_s)_0$, and $e_0,e_1,\dots,e_r$ form a basis of $(W_s)_2$. Denote $(m_{ij})$ as the $r\times r$ matrix defined via $N_ie_j = m_{ij}e_0$. Then $(m_{ij})$ is invertible.
    \end{enumerate}
\end{definition}

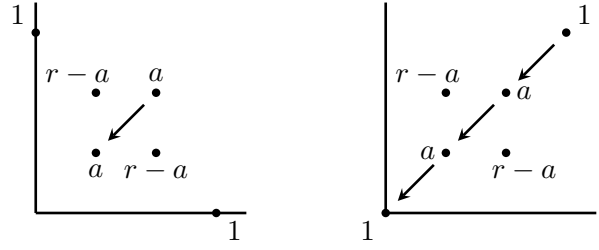
\begin{figure}

\centering
 \begin{tikzpicture}[scale=0.8]
    \draw[-,line width=1.0pt] (0,0) -- (0,3.5);
    \draw[-,line width=1.0pt] (0,0) -- (3.5,0);
    \fill (0,3) circle (2pt);
    \node at (-0.3,3.3) {$1$};

    \fill (1,2) circle (2pt);
    \node at (0.7,2.3) {$r-a$};
    
    \fill (1,1) circle (2pt);
    \node at (1,0.7) {$a$};
    
    \fill (3,0) circle (2pt);
    \node at (3.3,-0.3) {$1$};

     \fill (2,1) circle (2pt);
     \node at (2,0.7) {$r-a$};

    \fill (2,2) circle (2pt);
    \node at (2,2.3) {$a$};
    
    \draw [-stealth](1.8,1.8) -- (1.2,1.2) [line width = 1.0pt];
\end{tikzpicture}
    $
    \mspace{50mu}
    $
\begin{tikzpicture}[scale=0.8]
    \draw[-,line width=1.0pt] (0,0) -- (0,3.5);
    \draw[-,line width=1.0pt] (0,0) -- (3.5,0);
    \fill (0,0) circle (2pt);
    \node at (-0.3,-0.3) {$1$};
    
    \fill (1,1) circle (2pt);
    \node at (0.7,1) {$a$};

     \fill (1,2) circle (2pt);
    \node at (0.7,2.3) {$r-a$};
    
    \fill (3,3) circle (2pt);
    \node at (3.3,3.3) {$1$};
    
    \fill (2,1) circle (2pt);
    \node at (2.3,0.7) {$r-a$};
    
    \fill (2,2) circle (2pt);
    \node at (2.3,2) {$a$};
    
    \draw [-stealth](2.8,2.8) -- (2.2,2.2) [line width = 1.0pt];
    \draw [-stealth](1.8,1.8) -- (1.2,1.2) [line width = 1.0pt];
    \draw [-stealth](0.8,0.8) -- (0.2,0.2) [line width = 1.0pt];
    \end{tikzpicture}
    \caption{Type $\mathrm{I}_a$ and $\mathrm{IV}_a$ LMHS for $0\leq a\leq r$}
\end{figure}
\begin{figure}
\centering
\begin{tikzpicture}[scale=0.8]
    \draw[-,line width=1.0pt] (0,0) -- (0,3.5);
    \draw[-,line width=1.0pt] (0,0) -- (3.5,0);
    \fill (1,3) circle (2pt);
    \node at (1,3.3) {$1$};

    \fill (0,2) circle (2pt);
    \node at (-0.3,2) {$1$};
    
    \fill (1,1) circle (2pt);
    \node at (1,0.7) {$a$};

     \fill (1,2) circle (2pt);
    \node at (1,2.3) {$b$};
    
    \fill (3,1) circle (2pt);
    \node at (3.3,1) {$1$};

     \fill (2,0) circle (2pt);
     \node at (2,-0.3) {$1$};
    
    \fill (2,1) circle (2pt);
    \node at (2,0.7) {$b$};
    
    \fill (2,2) circle (2pt);
    \node at (2,2.3) {$a$};
    
    \draw [-stealth](1.8,1.8) -- (1.2,1.2) [line width = 1.0pt];
    \draw [-stealth](0.8,2.8) -- (0.2,2.2) [line width = 1.0pt];
    \draw [-stealth](2.8,0.8) -- (2.2,0.2) [line width = 1.0pt];
\end{tikzpicture}
    $
    \mspace{50mu}
    $
\begin{tikzpicture}[scale=0.8]
    \draw[-,line width=1.0pt] (0,0) -- (0,3.5);
    \draw[-,line width=1.0pt] (0,0) -- (3.5,0);
    \fill (2,3) circle (2pt);
    \node at (2,3.3) {$1$};

    \fill (0,1) circle (2pt);
    \node at (-0.3,1) {$1$};
    
    \fill (1,1) circle (2pt);
    \node at (1,0.7) {$a$};

     \fill (1,2) circle (2pt);
    \node at (1,2.3) {$b$};
    
    \fill (3,2) circle (2pt);
    \node at (3.3,2) {$1$};

     \fill (1,0) circle (2pt);
     \node at (1,-0.3) {$1$};
    
    \fill (2,1) circle (2pt);
    \node at (2,0.7) {$b$};
    
    \fill (2,2) circle (2pt);
    \node at (2,2.3) {$a$};
    
    \draw [-stealth](1.8,1.8) -- (1.2,1.2) [line width = 1.0pt];
    
    \draw [-stealth](1.8,2.8) -- (1.2,2.2) [line width = 1.0pt];
    \draw [-stealth](0.8,1.8) -- (0.2,1.2) [line width = 1.0pt];
    
    \draw [-stealth](2.8,1.8) -- (2.2,1.2) [line width = 1.0pt];
    \draw [-stealth](1.8,0.8) -- (1.2,0.2) [line width = 1.0pt];
    \end{tikzpicture}
    \caption{Type $\mathrm{II}_a$ and $\mathrm{III}_a$ LMHS for $a+b=h-1$}
\end{figure}
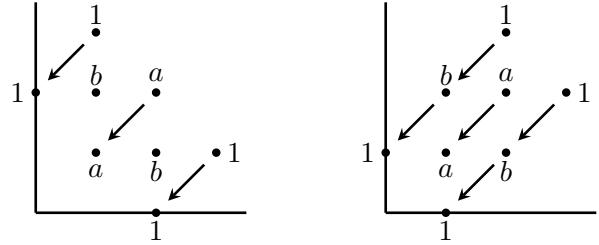

\section{Generic degree of two-parameter period mappings}\label{sec03}

In this section we assume $S$ is an algebraic surface admitting a projective compactification $\hat{S}$ which is smooth and $\hat{S}-S$ is a normal crossing divisor. $\mathbb{V}\rightarrow S$ is a PVHS with associated period map $\Phi: S\rightarrow \Gamma\backslash D$. We also assume the monodromy operator arouond each irreducible boundary divisor is unipotent which can always be done after a finite base change.

The main theorem of \cite{DR23} provides a way to realize $\Phi$ as the restriction of a morphism between compact complex analytic spaces:

\begin{theorem} \label{Thm:DR23main}
Suppose $\Gamma$ is neat, there exists a smooth compactification $\overline{S} \supset S$ with simple normal crossing divisor $\partial S = \overline{S} \backslash S$, and a logarithmic manifold $\Gamma \backslash D_\Sigma$ parameterizing $\Gamma$--conjugacy classes of nilpotent orbits on $D$ so that $\Gamma \bs D \subset \Gamma \backslash D_\Sigma$ and the period map extends to a morphism $\Phi_\Sigma : \overline{S} \to \Gamma \backslash D_\Sigma$ of logarithmic manifolds.  The image $\Phi_\Sigma(\overline{S})$ is a compact algebraic space. 
\end{theorem}
\begin{remark}
    More precisely, $\olS$ can be obtained by a (finite) sequence of blow-ups of $\hat{S}$ along codimensional $2$ boundary strata. In general for any given $S$ and $\hat{S}$, it is unrealistic to identify $\olS$ exactly.
\end{remark}

The image of $\Phi_\Sigma(s)$ for $s\in \olS-S$ can be described as follows. Let $\Sigma$ be the corresponding weak fan, $(\sigma_s, F_s)$ mod $\Gamma$ be the nilpotent orbit associated to $s$ by $\Phi$ via Schmid's nilpotent orbit theorem. There is a unique minimal $\tau_s\in \Sigma$ such that $\sigma_s\subset \tau_s$ and $(\tau_s, F_s)$ is a nilpotent orbit.

Let $\wp:=\mathrm{Img}(\Phi)\subset \Gamma\backslash D$, it is well-known that $\wp$ is a complex analytic space (and quasi-projective by \cite{BBT22}). Moreover, Theorem \ref{Thm:DR23main} implies there exists a compact complex analytic space $\wp_\Sigma\subset \Gamma\backslash D_\Sigma$ such that we have the following diagram in the category of complex analytic spaces:
\begin{equation}
\begin{tikzcd}
S \arrow[d] \arrow[r, "\Phi"] & \wp \arrow[d] \\
\overline{S} \arrow[r, "\Phi_\Sigma"] & \wp_\Sigma
\end{tikzcd}
\end{equation}

\subsection{A degree-computing formula}

Suppose $\wp$ has dimension $2$, then we have a well-defined degree for $\Phi$ and $\Phi_\Sigma$ defined to be $\mathrm{Card}\{\Phi^{-1}(p)\}$ for a generic $p\in \wp$, and $\mathrm{deg}(\Phi)=\mathrm{deg}(\Phi_\Sigma)$. The advantage of computing $\mathrm{deg}(\Phi_\Sigma)$ instead of $\mathrm{deg}(\Phi)$ is that we will be able to look at special boundary points.

\begin{prop}\label{Prop:DegComp}
    Suppose $s\in \partial S$ with $\sigma_s$ has dimension $2$, and moreover $s$ is not a branch point for $\Phi_\Sigma$, then $\mathrm{deg}(\Phi)=\mathrm{deg}(\Phi_\Sigma)$ equals to $\mathrm{Card}\{\Phi^{-1}(\Phi(s))\}$.
\end{prop}

In general $\Gamma$ is not neat, and we will need to pass to a neat subgroup $\Gamma^{'}\leq \Gamma$ of finite index. There exists a finite analytic covering map $\varphi: S^{'}\rightarrow S$ such that 
\begin{equation}
\begin{tikzcd}
\pi_1(S^{'}) \arrow[d] \arrow[r, "\rho"] & \Gamma^{'} \arrow[d] \\
\pi_1(S) \arrow[r, "\rho"] & \Gamma
\end{tikzcd}
\end{equation}
where $\rho$ is the monodromy representation. Lift the original PVHS to $S^{'}$:
\begin{equation}\label{Fig:liftbyneatsubgroup}
\begin{tikzcd}
S^{'} \arrow[d, "\varphi"] \arrow[r, "\Phi^{'}"] & \Gamma^{'}\backslash D \arrow[d] \\
S \arrow[r, "\Phi"] & \Gamma\backslash D
\end{tikzcd}
\end{equation}
Let $\wp^{'}:=\mathrm{Img}(\Phi^{'})$. Since both $S^{'}\xrightarrow{\varphi} S$ and $\wp^{'}\rightarrow \wp$ have the same degree $=[\Gamma: \Gamma^{'}]$, $\mathrm{deg}(\Phi^{'})=\mathrm{deg}(\Phi)$. In other words, passing $\Gamma$ to a finite-index neat subgroup does not change the generic degree. 

Take any $s\in \olS-S$ with $\mathrm{dim}(\sigma_s)=2$. Suppose for any $t\neq s$, $\sigma_s \neq \sigma_t$ mod $\Gamma$. Take a small neighborhood $s\in U\subset \olS$ with $U\cap S\simeq (\Delta^{*})^2$. The restricted period map $\Phi_{U\cap S}$ can be decomposed as:
\begin{equation}\label{eqn:local2globalperiodmap}
    U\cap S\xrightarrow{\Phi_s} \Gamma_s\backslash D \xrightarrow{\varphi_s} \Gamma\backslash D
\end{equation}
where $\Gamma_s$ is generated by the local monodromy around $s$ and $\varphi_s$ is the projection map. 

Let $\wp_s:=\Phi_s(U\cap S)$. Since $\sigma_s$ is non-degenerate, possibly after shrinking $U$ we may assume $\Phi_s$ and $\varphi_s|_{\wp_s}$ are both proper morphisms between complex analytic spaces such that the (generic) degree of them are well-defined.
\begin{prop}
    $\mathrm{deg}(\Phi)=\mathrm{deg}(\Phi_s)\cdot\mathrm{deg}(\varphi_s|_{\wp_s})$, 
\end{prop}
\begin{proof}
    This follows immediately from the sequence \ref{eqn:local2globalperiodmap} and our assumption on $\sigma_s$. 
\end{proof}
Note that by \cite{KU08}, $\Phi_s$ admits a Kato-Usui type completion
\begin{equation}
    \overline{\Phi_s}: U\rightarrow \Gamma_s\backslash D_s.
\end{equation}
The assumption $\mathrm{dim}(\sigma_s)=2$ implies $\Phi_s$ is proper up to a possible shrinking of $U$, therefore by \cite{Usu06}, $\overline{\wp_s}:=\overline{\Phi_s}(U)$ admits a structure of complex analytic space. Therefore we also have:
\begin{prop}\label{prop:localtoglobaldegree}
    $\mathrm{deg}(\Phi)=\mathrm{deg}(\overline{\Phi_s})\cdot\mathrm{deg}(\overline{\varphi_s}|_{\overline{\wp_s}})$, 
\end{prop}
\begin{remark}
    In $\ref{prop:localtoglobaldegree}$, $s$ is chosen in the unknown space $\olS$. However, since $\olS$ and $\hat{S}$ are birationally equivalent, we may apply Proposition \ref{prop:localtoglobaldegree} on the original space $\hat{S}$.
\end{remark}
To use the formula, we need to compute the two degrees on the right hand side separately. $\mathrm{deg}(\overline{\Phi_s})$ will be computed via coordinate interpretation and we will do it for specific examples in Section 4-5. In the next subsection we show how to find $\mathrm{deg}(\overline{\varphi_s}|_{\overline{\wp_s}})$.

\subsection{The local-to-global map for Kato-Nakayama-Usui spaces}

The following two lemmas are critical to our argument:

\begin{lemma}[\cite{KU08}, Prop. 7.4.3]
    Suppose $\Gamma\leq G_\bZ$ is neat and $(\sigma, F)$ is a nilpotent orbit. If $\gamma\in \Gamma$ satisfies $\gamma\cdot (\sigma, F)=(\sigma, F)$, then $\gamma\in \Gamma_\sigma:=\exp(\sigma_\bC)\cap \Gamma$.
\end{lemma}
\begin{lemma}[\cite{KU08}, Thm. A(iv)]
    Assume $\Gamma$ is neat, then $\Gamma_\sigma\backslash D_\sigma\hookrightarrow\Gamma\backslash D_\Sigma$ is a local homeomorphism.
\end{lemma}

These two lemmas and the diagram \eqref{Fig:liftbyneatsubgroup} imply that to calculate $\mathrm{deg}(\overline{\varphi_s}|_{\overline{\wp_s}})$ for non-neat $\Gamma$ (with a finite-index neat subgroup $\Gamma^{'}$ chosen), it is enough to consider the set
\begin{center}
    $\Gamma_{\sigma, \mathrm{tor}}:=\{\gamma\in \Gamma-\Gamma^{'} , \ \gamma (\sigma_s, F_s)=(\sigma_s, F_s), \ \gamma^n\in \Gamma_\sigma \ \text{for some} \ n\in \bZ\}$
\end{center}
and its induced automorphism group on the nilpotent orbit $(\sigma_s, F_s)$. 

Suppose $\sigma_s=\langle N_1, N_2\rangle$. Any finite-order automorphism $\eta$ of the nilpotent orbit $(\sigma_s, F_s)$ is a combination of the following actions:
\begin{enumerate}
    \item Rescaling $N_i$ by a root of unity, and fixes some chosen base point $F_s$;
    \item Permutes $N_1$ and $N_2$, and fixes some chosen base point $F_s$.
\end{enumerate}
In the case $\eta=\mathrm{Ad}_\mu$ for some $\mu\in \Gamma_{\sigma, \mathrm{tor}}$, it must not rescale $N_i$ by a root of unity other than $1$ because of rationality and positivity. Therefore, the only case $\eta$ could be non-trivial is $\eta$ flips the cone. Combining with the fact that $\mathrm{Ad}_\mu$ does not change the LMHS type, we have:
\begin{prop}\label{prop:localtoglobalKNUmapdegree}
  $\mathrm{deg}(\overline{\varphi_s}|_{\overline{\wp_s}})=1$ or $2$. It is $2$ if and only if there exists $\mu\in \Gamma$ whose adjoint action on $(\sigma_s, F_s)$ preserves the nilpotent orbit but flips the boundary of the cone. In particular, if $(\sigma_s, F_s)$ has LMHS type $\langle A|B|C \rangle$ with $A\neq C$, we must have $\mathrm{deg}(\overline{\varphi_s}|_{\overline{\wp_s}})=1$.
\end{prop}
\begin{remark}
    The same arguments show for a nilpotent orbit $(\sigma, F)$ with $\mathrm{dim}(\sigma)=n$, its automorphism group induced from any arithmetic group $\Gamma\in G_\bQ$ is isomorphic to a subgroup of $S_n$.
\end{remark}

\section{Examples: Calabi-Yau 3-folds in toric variety}\label{sec04}

In this section, we will study the geometric VHS coming from the tautological family of Calabi-Yau hypersurface inside the toric variety over the simplified moduli $\cM_{\mathrm{simp}}$.
The results in this section will be used to compute the generic degree of the period maps
\[
\Phi_{\mathrm{simp}}:\cM_{\mathrm{simp}}\rightarrow \mathrm{\Gamma}\backslash D.
\]
Moreover, we will also compute the generic degree of the map
\[
\phi:\cM_{\mathrm{simp}} \rightarrow \cM.
\]
Then the generic Torelli theorem for the Calabi-Yau hypersurface in question will be established if
\[
\deg(\Phi_{\mathrm{simp}}) = \deg(\phi).
\]

The subsections will be named after the Hodge type of the Calabi-Yau 3-folds, namely in subsection $V_{(a,b)}$, we will study a Calabi-Yau 3-fold with Hodge numbers $(h^{2,1},h^{1,1}) = (a,b)$. These Hodge numbers can be easily computed from the toric data by 
\cite{Bat94}.


\subsection{$V_{(2,29)}$}
We study the smooth Calabi-Yau 3-folds $V$ with Hodge numbers $(h^{2,1},h^{1,1}) = (2,29)$, which is the anti-canonical hypersurface of the toric variety $X_{\Sigma^\circ}\cong \bP_{\Delta^\circ}$. The toric data for $\Delta^\circ$ and $\Sigma$ is
\begin{equation}\label{eqn5.1}
\begin{array}{c|cccccc}
\Xi&v_1 & v_2 & v_3 & v_4 & v_5 & v_6\\
\hline
r_1 & 1 & 0 & 0 & 0 & 1 & -2\\
r_2 & 0 & 1 & 0 & -1 & 1 & -1\\
r_3 & 0 & 0 & 1 & -1 & 2 & -2\\
r_4 & 0 & 0 & 0 & 0 & 3 & -3\\
\hline
w_1 & 1 & 0 & 0 & 0 & 1 & 1\\
w_2 & 0 & 1 & 1 & 1 & 0 & 0
\end{array}.
\end{equation}
The polytope $\Delta^\circ$ can be visualized as follows: $\Xi$ is naturally divided into two sets $\{v_1,v_5,v_6\}$ and $\{v_2,v_3,v_4\}$ where the points in both sets lie in the same plane, the 9 facets of $\Delta$ then come from choosing two points from each set.
After suspension, the toric data becomes
\begin{equation}\label{susfan5.2}
\begin{array}{c|ccccccc}
\bar\Xi& \bar v_1 & \bar v_2 & \bar v_3 &\bar  v_4 &\bar  v_5 &\bar  v_6 &\bar  v_7\\
\hline
\bar r_1 & 1 & 0 & 0 & 0 & 1 & -2 & 0\\
\bar r_2 & 0 & 1 & 0 & -1 & 1 & -1 & 0\\
\bar r_3 & 0 & 0 & 1 & -1 & 2 & -2 & 0\\
\bar r_4 & 0 & 0 & 0 & 0 & 3 & -3 & 0\\
\bar r_5 & 1 & 1 & 1 & 1 & 1 & 1 & 1\\
\hline
\bar w_1 & 1 & 0 & 0 & 0 & 1 & 1 & -3\\
\bar w_2 & 0 & 1 & 1 & 1 & 0 & 0 & -3
\end{array}.
\end{equation}

\subsubsection{Moduli spaces}\label{SubSec:4.1.1}
The toric data shows that $(\Delta^\circ \cap N)_0 = \Delta^\circ \cap N$, i.e., there is no integral point lies in the interior of codimension 2 face. Therefore, we have $\cM = \cM_{\mathrm{poly}}$.

From the multi-degree, it is clear that $\bP_{\Delta^\circ}$ has no root symmetry. Then the generic degree of the map $\phi:\cM_{\mathrm{poly}} \rightarrow  \cM_{\mathrm{simp}}$
is $|\tAut(\Sigma)/\tAut^{t}(\Sigma)|$, where $\tAut^{t}(\Sigma)$ is the normal subgroup of $\tAut(\Sigma)$ that acts trivially on $\cM_{\mathrm{poly}}$. 

Denote $\tAut(\Xi)$ as the subgroup of $\tAut(N)$ that permutes $\Xi$, then a direct calculation shows $\tAut(\Xi) = C_2\times C_3\rtimes S_3$, which acts on $\Xi$ as follows: the $C_2$ factor determines whether the permutation is even or odd, the set $\{v_2,v_3,v_4\}$ is sent to $\{v_2,v_3,v_4\}$ or $\{v_1,v_5,v_6\}$ respectively for the two cases. Then the $S_3$ factor determines the image of $v_2,v_3,v_4$ under the automorphism, and since the parity of the permutation is already determined, there are only $C_3$ choices for the image of $v_1,v_5,v_6$. From this description, it is also clear that every element in $\tAut(\Xi)$ permits the codimension-1 cones of $\Sigma$, thus we have
$\tAut(\Sigma) = \tAut(\Xi)$. As a subgroup of $S_6$ which acts on $\Xi$ by permutation, we have $\tAut(\Sigma)$ is generated by the permutations $a:= (234),b:=(34)(56),c:=(12)(35)(46)$.

A generic point in $\bP(L(\Delta^\circ\cap N))$ is given by
\[
f = \lambda_1 t_1+\lambda_2t_2+\lambda_3t_3+\lambda_4 t^{-1}_2t_3^{-1}+\lambda_5t_1t_2t_3^2t_4^3+\lambda_6t_1^{-2}t_2^{-1}t_3^{-2}t_4^{-3}+\lambda_7,
\]
Then we have
\[
a\cdot f = \lambda_1 t_1+\lambda_3t_2+\lambda_4t_3+\lambda_2 t^{-1}_2t_3^{-1}+\lambda_5t_1t_2t_3^2t_4^3+\lambda_6t_1^{-2}t_2^{-1}t_3^{-2}t_4^{-3}+\lambda_7,
\]
which can be realized as a $T$-action 
\[
(t_1,t_2,t_3,t_4)\rightarrow (t_1,\lambda_2^{-1}\lambda_3t_2,\lambda_3^{-1}\lambda_4t_3,\lambda_2^{1/3}\lambda_3^{1/3}\lambda^{-2/3}_4t_4),
\]
similarly, $b$ can be realized as
\[
(t_1,t_2,t_3,t_4)\rightarrow (t_1,t_2,\lambda_3^{-1}\lambda_4t_3,\lambda_3^{2/3}\lambda^{-2/3}_4\lambda_5^{-1/3}\lambda_6^{1/3}t_4).
\]
However, $c$ can not be realized as a $T$-action, therefore, we have $\tAut^t(\Sigma) = C_3\rtimes S_3$, and the generic degree of $\phi$ is
\[
\deg_g(\phi)=|\tAut(\Sigma)/\tAut^t(\Sigma)|=2.
\]

The toric data for the secondary fan $\Sigma^s$ is
\begin{equation}
\begin{array}{c|ccc}
\Xi^s&v_1^s & v_2^s & v_3^s \\
\hline
r_1^s & 1 & 0 & -1 \\
r_2^s & 0 & 1 & -1 \\
\hline
w_1^s & 1 & 1 & 1 \\
\end{array}.
\end{equation}
Therefore, $\overline{\cM}_{\mathrm{simp}} = X_{\Sigma^s} \cong \bP^2 = \mathrm{Proj}\bC[Z_1,Z_2,Z_3]$, and the $\bZ/2\bZ$-action extend to $\overline{\cM}_{\mathrm{simp}}$ as $Z_1\leftrightarrow Z_2$ that ramifies at $Z_1=Z_2$.

\subsubsection{Picard-Fuchs system}\label{SubSec:4.1.2}
By Lemma \ref{GKZsys}, the GKZ system is generated by
\begin{align*}
    & P_1:= \delta_1^3+z_1(3\delta_1+3\delta_2+1)(3\delta_1+3\delta_2+2)(3\delta_1+3\delta_2+3),\\
    & P_2:= \delta_2^3+z_2(3\delta_1+3\delta_2+1)(3\delta_1+3\delta_2+2)(3\delta_1+3\delta_2+3).
\end{align*}
The canonical affine chart is $\tSpec[z_1,z_2]$, with
\[
z_1 = \frac{\lambda_1\lambda_5\lambda_6}{\lambda_7^2},\quad z_2 = \frac{\lambda_2\lambda_3\lambda_4}{\lambda_7^2},
\]
then $D_{v_1^s} = \bV(z_1)$, $D_{v_{1^s}} = \bV(z_2)$. In the projective coordinates, this is the chart with $Z_3\neq 0$, and $z_i = \frac{X_i}{X_3}, i=1,2$.

By Lemma, \ref{lemquotmum} and the comment after \eqref{eqn:monomatricesv229}, the Picard-Fuchs system on $\cM_{\mathrm{simp}}$ is generated by the two differential operators
\begin{align*}
    & P_1:= \delta_1^3+z_1(3\delta_1+3\delta_2+1)(3\delta_1+3\delta_2+2)(3\delta_1+3\delta_2+3),\\
    & P_2:= (\delta_1^2-\delta_1\delta_2+\delta_2^2)+3(z_1+z_2)(3\delta_1+3\delta_2+1)(3\delta_1+3\delta_2+2).
\end{align*}
For simplicity, we rescale the local coordinates as $z_{1,2}':=3^3\cdot z_{1,2}$, and rename $z_{1,2}'$ as the new $z_{1,2}$. The Picard-Fuchs operators then become
\begin{equation}\label{eqp1p2}
\begin{aligned}
    & P_1:= \delta_1^3+z_1(\delta_1+\delta_2+\frac13)(\delta_1+\delta_2+\frac23)(\delta_1+\delta_2+1),\\
    & P_2:= (\delta_1^2-\delta_1\delta_2+\delta_2^2)+(z_1+z_2)(\delta_1+\delta_2+\frac13)(\delta_1+\delta_2+\frac23).
\end{aligned}
\end{equation}
The Picard-Fuchs ideal sheaf is denoted as $I$, which in the canonical affine affine chart is generated by $P_1$ and $P_2$. The associated $\cD$-module is $\bZ/2\bZ$-equivariant under $z_1\leftrightarrow z_2$, $\partial_1\leftrightarrow \partial_2$. More precisely, $P_2$ is invariant under the $\bZ/2\bZ$-action, and $P_1$ changes to $(\delta_1+\delta_2)P_2-P_1$, so $I$ is $\bZ/2\bZ$-invariant.

\subsubsection{Discriminant locus}

The principal $A$-determinant $E_A$ is irreducible, and is thus equals to the $A$-discriminant $D_A$
\begin{align*}
D_A = & z_1^3 + 3z_1^2z_2 + 3z_1^2 + 3z_1z_2^2 - 21z_1z_2 + 3z_1 + z_2^3 + 3z_2^2 + 3z_2 + 1\\
 = & (z_1+z_2+1)^3-27z_1z_2
\end{align*}
so the discriminant locus is contained inside
\[
Disc= D_A\cup (\bar B- B) = D_A\cup D_{v_1^s}\cup D_{v_2^s}\cup D_{v_3^s}.
\]

In the canonical affine chart $D_A$ intersect $D_{v_1^s}$ at $(-1,0)$ with 3-tangency, and by the $\bZ/2\bZ$ symmetry, $D_A$ intersect $D_{v_2^s}$ at $(0,-1)$ with 3-tangency.

In order to make $Disc$ only have normal crossing singularities, we need to blow-up these tangencies 3 times each. Since we will only be interested in the local monodromy, we only need to study the geometry at $z_1 = 0, z_2 = -1$. Denote the exceptional divisors of the three blow-ups as $E_1,E_2,E_3$, then we want to find 4 Picard-Fuchs ideals corresponding to the normal crossings 
\[
D_{v_1}\cap E_3, \quad E_1\cap E_2,\quad E_2\cap E_3\quad D_A\cap E_3,
\]
First, we make a change of coordinate of \eqref{eqp1p2} by $z_1 \rightarrow z_1$, $z_2\rightarrow z_2-1$. Then the local affine coordinates systems for each of the normal crossing is given by
\begin{equation*}
 \begin{aligned}
     & (s,t) = \left(\frac{z_1}{z_2^3},z_2\right),\quad && (s,t)= \left(\frac{z_2^2}{z_1},\frac{z_1}{z_2}\right),\\
     & (s,t)= \left(\frac{z_2^3}{z_1},\frac{z_1}{z_2^2}\right), \quad && (s,t) = \left(\frac{z_1}{z_2^3}-\frac{1}{27},z_2\right)
 \end{aligned}   
\end{equation*}
With $s=0$ gives the left divisor and $t=0$ gives the right divisor.

\subsubsection{Gauss-Manin connection and nilpotent cones}
By the local Torrelli theorem of Calabi-Yau manifolds, we can take a global multi-valued frame 
\begin{equation}\label{eqn:sympbasisV229}
w = (\Omega,\delta_1\Omega,\delta_2\Omega,\delta_1\delta_2\Omega,\delta_1^2\Omega,\delta_1^2\delta_2\Omega).
\end{equation}

The Gauss-Manin connection can be easily calculated from the Picard-Fuchs ideal by finding the 
\[
\begin{split}
&\nabla_{\delta_1} w =
w \left[
\begin{array}{rrrrrr}
     0 & 0 & 0 & 0 & z_1h^1_{1,5} & z_1h^1_{1,6} \\
     1 & 0 & 0 & 0 & z_1h^1_{2,5} & z_1h^1_{2,6}  \\
     0 & 0 & 0 & 0 & z_1h^1_{3,5} & z_1h^1_{3,6} \\
     0 & 0 & 1 & 0 & z_1h^1_{4,5} & z_1h^1_{4,6} \\
     0 & 1 & 0 & 0 & z_1h^1_{5,5} & z_1h^1_{5,6} \\
     0 & 0 & 0 & 1 & z_1h^1_{6,5} & z_1h^1_{6,6}
\end{array}
\right]=:w\cdot R_1,\\
&\nabla_{\delta_2} w =
w \left[
\begin{array}{rrrrrr}
     0 & 0 & z_2h^2_{1,3} & z_2h^2_{1,4} & 0 & z_2h^2_{1,6}  \\
     0 & 0 & z_2h^2_{2,3} & z_2h^2_{2,4} & 0  & z_2h^2_{2,6} \\
     1 & 0 & z_2h^2_{3,3} & z_2h^2_{3,4} & 0 & z_2h^2_{3,6}  \\
     0 & 1 & 1+z_2h^2_{4,3} & z_2h^2_{4,4} & 0  & z_2h^2_{4,6} \\
     0 & 0 & -1 & z_2h^2_{4,4}& 0  & z_2h^2_{5,6} \\
     0 & 0 & 0 & 1+z_2h^2_{5,4} & 1 & z_2h^2_{6,6}
\end{array}
\right]=:w\cdot R_2,
\end{split}
\]
 where $h^{i}_{l,m}(z_1,z_2)$ are rational functions that is holomorphic at $z_1=z_2=0$. 

Denote $T_1,T_2$ as the log-monodromy matrix around a small loop of the divisors $z_1=0$ and $z_2=0$, represented on $F^0_{\mathrm{lim}}$. Denote the (Tate twisted) log-monodormy matrix as $\overline N_i = \frac{1}{2\pi i}\log(T_i)$, then we have $\overline N_i = (\mathrm{Res}_{z_i=0}R_i)|_{z=0}$, and more precisely

\begin{equation}\label{eqn:monomatricesv229}
\overline N_1 =\left[
\begin{array}{cccccc}
    0 & 0  & 0 & 0 & 0 & 0  \\ 
    1 & 0  & 0 & 0 & 0 & 0 \\ 
    0 & 0 &  0&  0&  0&  0\\ 
    0 & 0 & 1 & 0 & 0 & 0  \\ 
    0 & 1 & 0 & 0 & 0 & 0  \\ 
    0 & 0 & 0 & 1 & 0 &  0 \\ 
\end{array}
\right],\quad \overline N_2 =\left[
\begin{array}{cccccc}
    0 & 0  & 0 & 0 & 0 & 0  \\ 
    0 & 0  & 0 & 0 & 0 & 0 \\ 
    1 & 0 &  0&  0&  0&  0\\ 
    0 & 1 & 1 & 0 & 0 & 0  \\ 
    0 & 0 & -1 & 0 & 0 & 0  \\ 
    0 & 0 & 0 & 1 & 1 &  0 \\ 
\end{array}
\right].
\end{equation}
Then the nilpotent cone associated to $z_1=z_2=0$ is generated by $\overline N_1,\overline N_2$, and a simple calculation shows that this is a MUM point. Therefore, the Picard-Fuchs system \eqref{eqp1p2} is indeed the one we want by Lemma \ref{lemquotmum}.

\subsubsection{Local monodromy and LMHS type}

By the same method as in \cite{Chen23}, the types of LMHS of the rest possible two-dimensional nilpotent cones can be determined up to some ambiguity by calculating the Jordan normal form of the nilpotent operators of the boundary and using the polarization relations. The result is
\begin{align}
 \label{V229:cone1}    \mathrm{Cone}(D_{v_1^s},D_{v_2^s}) & = \braket{III_0|IV_2|III_0}\\
 \label{V229:cone2}   \mathrm{Cone}(D_{v_1^s},E_3) &  = \braket{III_0|IV_2|IV_1} \\ 
    \mathrm{Cone}(E_1,E_2) & = \braket{IV_1|IV_2(IV_1)|IV_1}\\
    \mathrm{Cone}(E_2,E_3) & = \braket{IV_1|IV_2(IV_1)|IV_1}\\
 \label{V229:cone5}    \mathrm{Cone}(D_A,E_3) & = \braket{I_1|IV_2|IV_1}
\end{align}
where there is another copy of cone \eqref{V229:cone2}-\eqref{V229:cone5} by the $\bZ_2$-symmetry of the $\cD$-module. The possible ambiguity has been indicated inside the parathesis. Since the monodromy around $D_{v_3^s}$ is finite, we have excluded the cones associated with it.

\subsubsection{Yukawa coupling and symplectic form}\label{V229:sec.Yuka}
In this subsection, we will identify $\omega$ as 1.
The symplectic form $Q$ is determined by the intersection matrix $Q_{ij} = \int_{V}w_i\cup w_j$, which in turn is determined by the \textit{$n$-point functions}
\[
K^{ij} :=\int_{V} \Omega\wedge \delta_1^i\delta_2^j\Omega, \quad i+j=n.
\]
When $n=3$, these are the so-called \textit{Yukawa couplings}. Rewrite the Picard-Fuchs operators
\begin{align*}
     P_1:=& (1+z_1)\delta_1^3+3z_1\delta_1^2\delta_2+3z_1\delta_1\delta_2^2+z_1\delta_2^3+2z_1\delta_1^2+4z_1\delta_1\delta_2+2z_1\delta_2^2\\
    & +\frac{11}{9}z_1\delta_1+ \frac{11}{9}z_1\delta_2 +\frac29    z_1 \\
     P_2:= & (z_1 + z_2 + 1)\delta_1^2+(2z_1+2z_2-1)\delta_1\delta_2+(z_1 + z_2 + 1)\delta_2^2\\
     &+(z_1+z_2)\delta_1+(z_1+z_2)\delta_2+\frac29(z_1+z_2).
\end{align*}
Together with the Griffith transversality, from $\int_V\Omega\wedge P_1\Omega =0$ we get
\[
(1+z_1)K^{3,0}+3z_1K^{2,1}+3z_1K^{1,2}+z_1K^{0,3} = 0.
\]
Similarly, from $\int_V\Omega\wedge\delta_i P_2\Omega=0$, $i=1,2$, we get
\begin{align*}
    &(z_1 + z_2 + 1)K^{3,0}+(2z_1+2z_2-1)K^{2,1}+(z_1 + z_2 + 1)K^{1,2} =0,\\
    &(z_1 + z_2 + 1)K^{2,1}+(2z_1+2z_2-1)K^{1,2}+(z_1 + z_2 + 1)K^{0,3} =0.
\end{align*}
Therefore, $K^{3,0}$ determines all the other Yukawa couplings. More precisely, we have
\begin{align*}
    K^{21} & = \frac{- 2z_1^2 - z_1z_2 - z_1 + z_2^2 + 2z_2 + 1}{3z_1(z_1 + z_2 - 2)}K^{30},\\
    K^{12} & = \frac{- 2z_2^2 - z_1z_2 - z_2 + z_1^2 + 2z_1 + 1}{3z_1(z_1 + z_2 - 2)}K^{30},\\
    K^{03} & = \frac{z_2}{z_1}K^{30}.
\end{align*}
We note these results are consistent with the $\bZ/2\bZ$ symmetry. Next, from $\int_V \Omega\wedge \delta_1 P_1\Omega=0$, we have
\begin{equation}\label{V229eqn5.4}
\begin{split}
0 = & (1+z_1)K^{4,0}+3z_1K^{3,1}+3z_1K^{2,2}+z_1K^{1,3}+3z_1K^{3,0}\\
& +7z_1K^{2,1}+5z_1K^{1,2}+z_1K^{0,3}, 
\end{split}
\end{equation}
similarly, from $\int_{V}\Omega\wedge\delta_1^2P_2\Omega$, we have
\begin{equation}\label{V229eqn5.5}
\begin{split}
0 =  & (z_1+z_2+1)K^{4,0}+(2z_1+2z_2-1)K^{31}+(z_1+z_2+1)K^{2,2}\\
& + (3z_1+z_2)K^{3,0}+(5z_1+z_2)K^{2,1}+2z_1K^{1,2}.
\end{split}
\end{equation}

To simplify the notation, we denote $\braket{i,j|l,m}=\int_V\delta_1^i\delta_2^j\Omega\wedge \delta_1^l\delta_2^m\Omega$. Then $K^{i,j} = \braket{0,0|i,j}$, and we have $\braket{i,j|l,m} = -\braket{l,m|i,j}$. From $\delta_1^2K^{2,0}=0$, we have $K^{4,0} = -2\braket{1,0|3,0}$, and since $\delta_1K^{3,0} = \braket{1,0|3,0}+K^{4,0}$, we have $K^{4,0} = 2\delta_1K^{3,0}$. Similarly, from $\delta_1\delta_2K^{2,0}=\delta_1^2K^{1,1}=0$ we get
\begin{align*}
    &\braket{1,0|2,1}+\braket{0,1|3,0}+\braket{1,1|2,0}+K^{3,1} = 0,\\
    &2\braket{1,0|2,1}-\braket{1,1|2,0}+K^{3,1}=0.
\end{align*}
combining with
$\delta_1 K^{2,1} = \braket{1,0|2,1}+K^{3,1}$, and $\delta_2K^{3,0} = \braket{0,1|3,0}+K^{3,1}$, we have
\[
K^{3,1} = \frac{3}{2}\delta_1K^{2,1} +\frac12\delta_2K^{3,0}.
\]
By symmetry, we have
\[
K^{1,3} = \frac{3}{2}\delta_2K^{1,2} +\frac12\delta_1K^{0,3},\quad K^{0,4} = 2\delta_2K^{0,3}. 
\]
A similar calculation shows 
\[
K^{2,2} = \delta_1 K^{1,2}+\delta_2 K^{2,1},
\]
Then substitute everything back into \eqref{V229eqn5.4} and \eqref{V229eqn5.5}, we arrive at the system of first order PDEs
\begin{equation}
    \begin{split}
        &\delta_1K^{3,0} = \frac{\alpha_1}{(z_1+z_2-2)D_A}K^{3,0}\\
        &\delta_2K^{3,0} = \frac{\alpha_2}{(z_1+z_2-2)D_A}K^{3,0}
    \end{split}
\end{equation}
where 
\begin{align*}
    \alpha_1 = & - z_1^4 - 2z_1^3z_2 + 4z_1^3 - 18z_1^2z_2 + 9z_1^2 + 2z_1z_2^3 + 6z_1z_2^2  + 6z_1z_2 + 2z_1 \\
    & + z_2^4 + z_2^3 - 3z_2^2 - 5z_2 - 2,\\
    \alpha_2 = & -z_2(2z_1^3 + 6z_1^2z_2 - 30z_1^2 + 6z_1z_2^2 - 6z_1z_2 + 42z_1  + 2z_2^3 \\
    & - 3z_2^2 - 12z_2 - 7),
\end{align*}
With the help of Wolfram Mathematica$^{{\text{\tiny\textregistered}}}$ we get
\[
K^{3,0} = \frac{cz_1}{(z_1+z_2-2)D_A},\quad c\neq 0,
\]
Next, the intersection matrix is given by
\begin{equation}
    Q = \left[\begin{array}{cccccc}
    0 & 0 & 0 & 0 & 0 & K^{2,1} \\
    0 & 0 & 0 & \braket{1,0|1,1} & \braket{1,0|2,0} & \braket{1,0|2,1} \\
    0 & 0 & 0 & \braket{0,1|1,1} & \braket{0,1|2,0} & \braket{0,1|2,1} \\
    * & * & * & 0 & \braket{1,1|2,0} & \braket{1,1|2,1} \\
    * & * & * & * & 0 & \braket{2,0|2,1} \\
    * & * & * & * & * & 0 \\
    \end{array}\right]
\end{equation}
where the $*$ entries are determined by the rest by the antisymmetricity of $Q$. By the previous calculation, we have
\begin{equation}\label{QinYuk}
    Q = \left[\begin{array}{cccccc}
    0 & 0 & 0 & 0 & 0 & K^{2,1} \\
    0 & 0 & 0 & -K^{2,1} & -K^{3,0} & \delta_1K^{2,1}-K^{3,1} \\
    0 & 0 & 0 & -K^{1,2} & -K^{2,1} & \delta_2K^{2,1}-K^{2,2} \\
    * & * & * & 0 & 2\delta_1K^{2,1}-K^{3,1} & \braket{1,1|2,1} \\
    * & * & * & * & 0 & \braket{2,0|2,1} \\
    * & * & * & * & * & 0 \\
    \end{array}\right],
\end{equation}
Now, the entry $\braket{1,1|2,1}$ can be determined by the following. First, from
\begin{equation}
    \begin{aligned}
        0  = & \delta_1^3K^{3,0}  =  \braket{3,0|0,2}+3\braket{2,0|1,2}+3\braket{1,0|2,2}+K^{3,2},\\
        0  = &  \delta_1^2\delta_2K^{1,1}  =  \braket{2,1|1,1}+\braket{2,0|1,2}+2\braket{1,1|2,1}+2\braket{1,0|2,2}\\
        & +\braket{0,1|3,1} +K^{3,2}, \\
        \delta_1^2 K^{1,2} = & \braket{2,0|1,2}+ 2\braket{1,0|2,2}+K^{3,2},\\
        \delta_1\delta_2K^{2,1} = & \braket{1,1|2,1}+\braket{1,0|2,2}+\braket{0,1|3,1}+K^{3,2},\\
        \delta_2^2K^{3,0} = &\braket{0,2|3,0}+2\braket{0,1|3,1} + K^{3,2},
    \end{aligned}
\end{equation}
we get
\[
\braket{1,1|2,1} = K^{3,2}-\delta_1^2K^{1,2}  -\frac32\delta_1\delta_2K^{2,1}-\frac 12\delta_2^2 K^{3,0},
\]
and similarly one has
\begin{align*}
    \braket{2,0|2,1} = & K^{4,1}-\frac 43\delta_1K^{3,1}-\frac16\delta_2K^{4,0}\\
    = & K^{4,1} - 2 \delta_1^2K^{2,1}-\delta_1\delta_2K^{3,0}.
\end{align*}
Therefore, these two entries are determined by the 5-point functions $K^{3,2},K^{4,1}$. To calculate these 5-point functions, we consider $0 =\int_V1\wedge \delta_1^i\delta_2^j P_1$, $i+j=2$, and $0 =\int_V1\wedge \delta_1^i\delta_2^j P_2$, $i+j=3$. These 7 equations can be used to rewrite the 6 different 5-point functions into 4-point and 3-point functions. More precisely, the first 6 equations read
\begin{align*}
    0 = & (z_1 + 1)K^{5,0} + 3z_1K^{4,1} + 4z_1K^{4,0} + 3z_1K^{3,2} + 10z_1K^{3,1} + \frac{56z_1K^{3,0}}{9} \\ & + z_1K^{2,3}   + 8z_1K^{2,2} + \frac{110z_1K^{2,1}}{9} + 2z_1K^{1,3} + 7z_1K^{1,2} + z_1K^{0,3}\\
    0 = & (z_1 + 1)K^{4,1} + 3z_1K^{3,2} + 3z_1K^{3,1} + 3z_1K^{2,3} + 7z_1K^{2,2} \\& + \frac{29z_1K^{2,1}}{9} + z_1K^{1,4}  + 5z_1K^{1,3} + \frac{47z_1K^{1,2}}{9} + z_1K^{0,4} + 2z_1K^{0,3}\\
    0 = & (z_1 + 1)K^{3,2} + 3z_1K^{2,3} + 2z_1K^{2,2} + 3z_1K^{1,4} + 4z_1K^{1,3} \\ & + \frac{11z_1K^{1,2}}{9} + z_1K^{0,5}  + 2z_1K^{0,4} + \frac{11z_1K^{0,3}}{9}\\
    0 = & (z_1 + z_2 + 1)K^{5,0} + (2z_1 + 2z_2 - 1)K^{4,1} + (4z_1 + z_2)K^{4,0}   \\ & + (z_1 + z_2 + 1)K^{3,2} + (7z_1 + z_2)K^{3,1} + \frac{56z_1 + 2z_2}{9}K^{3,0}  \\ & + 3z_1K^{2,2} + 9z_1K^{2,1} + 3z_1K^{1,2}\\
    0 = & (z_1 + z_2 + 1)K^{4,1} + z_2K^{4,0} + (2z_1 + 2z_2 - 1)K^{3,2} + (3z_1 + 3z_2)K^{3,1} \\ &  + z_2K^{3,0}  +(z_1 + z_2 + 1)K^{2,3} + (5z_1 + 2z_2)K^{2,2} + \frac{29z_1 + 11z_2}{9}K^{2,1} \\
    & + 2z_1K^{1,3} + 4z_1K^{1,2} + z_1K^{0,3}\\
    0 = & (z_1 + z_2 + 1)K^{3,2} + 2z_2K^{3,1} + z_2K^{3,0} + (2z_1 + 2z_2 - 1)K^{2,3} \\ & + (2z_1 + 5z_2)K^{2,2} + 4z_2K^{2,1}  + (z_1 + z_2 + 1)K^{1,4} + (3z_1 + 3z_2)K^{1,3}\\
    & + \frac{11z_1 + 29z_2}{9}K^{1,2} + z_1K^{0,4} + z_1K^{0,3}
\end{align*}
and from these equations we can solve $K^{4,1}$ and $K^{3,2}$, which in turn gives $\braket{1,1|2,1}$ and $\braket{2,0|2,1}$. The result is
\begin{align*}
\braket{1,1|2,1} & = -\frac{cz_1\alpha_1}{27(z_1+z_2-2)^4D_A^2} ,\\
\braket{2,0|2,1} & =-\frac{cz_1\alpha_2}{27(z_1+z_2-2)^4D_A^2},
\end{align*}
where $\alpha_{1,2}\in \bZ[z_1,z_2]$ are two complicated degree 6 polynomials. Therefore, at $z_1=z_2=0$, the symplectic form is given by
\[
Q =\frac{c}{12}\cdot \left[
\begin{array}{cccccc}
    0 & 0  & 0 & 0 & 0 & 1  \\ 
    0 & 0  & 0 & -1 & 0 & 0 \\ 
    0 & 0 &  0&  -1&  -1&  0\\ 
    0 & 1 & 1 & 0 & 0 & 0  \\ 
    0 & 0 & 1 & 0 & 0 & 0  \\ 
    -1 & 0 & 0 & 0 & 0 &  0 \\ 
\end{array}
\right],
\]
By rescale the local basis $w$ by $\left(\frac{c}{12}\right)^{\frac12}$, we can get rid of the coefficient $\frac{c}{12}$. These calculations are done with the help of Matlab$^{{\text{\tiny\textregistered}}}$.

\subsection{$V_{(2,38)}$}
The toric data for this case is given by
\begin{equation}
\begin{array}{c|ccccccc}
\Xi&  v_1 &  v_2 &  v_3 &  v_4 &v_5 \\
\hline
 r_1 & 1 & 0 & 0 & 1 & -4 \\
 r_2 & 0 & 1 & 0 & 2 & -5 \\
 r_3 & 0 & 0 & 1 & 1 & -2 \\
 r_4 & 0 & 0 & 0 & 3 & -3 \\
\hline
 w_1 & 1 & 1 & 0 & 0 & 0 \\
 w_2 & 0 & 0 & 1 & 1 & 1 
\end{array}.
\end{equation}
Then, we have $X_\Sigma$ is a quotient of $\bP[3,3,1,1,1]$ by $\bZ/3\bZ$, and the Baytrev mirror of $X_\Sigma$ is $\bP_{\Delta^\circ}$. However, since $X_\Sigma$ has more than terminal singularity, we need to conduct a toric resolution by adding an additional column $(-1,-1,0,0)^t$. We will denote the new fan still as $\Sigma$. Then the suspended toric data is
\begin{equation}
\begin{array}{c|ccccccc}
\bar\Xi& \bar v_1 & \bar v_2 & \bar v_3 &\bar  v_4 &\bar  v_5 &\bar  v_6 &\bar  v_7\\
\hline
\bar r_1 & 1 & 0 & 0 & 1 & -4 & -1 & 0\\
\bar r_2 & 0 & 1 & 0 & 2 & -5 & -1 & 0\\
\bar r_3 & 0 & 0 & 1 & 1 & -2 & 0 & 0\\
\bar r_4 & 0 & 0 & 0 & 3 & -3 & 0 & 0\\
\bar r_5 & 1 & 1 & 1 & 1 & 1 & 1 & 1\\
\hline
\bar w_1 & 1 & 1 & 0 & 0 & 0 & 1 & -3\\
\bar w_2 & 0 & 0 & 1 & 1 & 1 & -3 & 0
\end{array}.
\end{equation}

\subsubsection{Moduli spaces}
Since $\Delta^\circ_{345}$ has an integral interior point $v_7$, in order to show $\cM_{\mathrm{poly}}=\cM$, we need to show the dual face of $\Delta^\circ_{345}$ has no integral interior point. A direct calculation gives that the dual face is the interval between $v^\circ_1,v^\circ_2\in M$, with $v^\circ_1 = (-1,2,-1,-1)$, and $v^\circ_2 = (2,-1,-1,0)$. Thus, we have $\cM_{\mathrm{poly}}=\cM$. It is also clear from the muti-degrees that $\tAut_r(\bP_{\Delta^\circ})$ is trivial.

The fan symmetry of $\tAut(\Sigma)$ the group $S_3$ which is generated by $(12)(34)$ and $(345)$.
 A generic point in $\bP(L(\Delta^\circ\cap N))$ is given by
\[
f = \lambda_1 t_1+\lambda_2t_2+\lambda_3t_3+\lambda_4 t_1t_2^2t_3t_4^3+\lambda_5t_1^{-4}t_2^{-5}t_3^{-2}t_4^{-3}+\lambda_6t_1^{-1}t_2^{-1}+\lambda_7,
\]
One can easily check both $(12)(34)$ and $(345)$ can be realized as a $T$ action, thus $\tAut(\Sigma)$ acts on $\cM_{\mathrm{simp}}$ trivially. Therefore, we have $\cM_{\mathrm{simp}}  = \cM$.

The secondary fan $\Sigma^s$ is
\begin{equation}
\frac{A^s}{B^s} = \begin{array}{c|cccc}
\Xi^s&v_1^s & v_2^s & v_3^s & v_4^s \\
\hline
r_1^s & 1 & 0 & 1 & -1\\
r_2^s & 0 & 1 & -3 & 0\\
\hline
w_1^s & 1 & 0 & 0 & 1 \\
w_2^s & 0 & 3 & 1 & 1 \\
\end{array}.
\end{equation}
Then $\overline{\cM}_{\mathrm{simp}} \cong \bP_{\Delta^s}$ is a smooth toric varies. More concretely, we have $\bP_{\Delta^s}\cong \mathrm{Proj}_{B^s}\bC[Z_1,Z_2,Z_3,Z_4]$, with $Z_i$ has multi-degree $(w_{1,i}^s,w_{2,i}^s)$. Then the canonical affine coordinates are
\[
z_1 = \frac{\lambda_1\lambda_2\lambda_6}{\lambda_7^3} = \frac{Z_1Z_3}{Z_4},\quad z_2 =  \frac{\lambda_3\lambda_4\lambda_5}{\lambda_6^3} = \frac{Z_2}{Z_3^3}.
\]

\subsubsection{Picard-Fuchs equations and discriminant locus}
\begin{align*}
    & P_1 = \delta_1(\delta_1-3\delta_2) +3z_1(3\delta_1+1)(3\delta_1+2)\\
    & P_2 = \delta_2^3-z_2(\delta_1-3\delta_2)(\delta_1-3\delta_2-1)(\delta_1-3\delta_2-2)
\end{align*}
For simplicity, we rescale the local coordinates as $z_{1,2}':=3^3\cdot z_{1,2}$, and rename $z_{1,2}'$ as the new $z_{1,2}$. The Picard-Fuchs operators become
\begin{align*}
    & P_1 = \delta_1(\delta_1-3\delta_2) +z_1(\delta_1+\frac13)(\delta_1+\frac23)\\
    & P_2 = \delta_2^3-z_2(\frac13\delta_1-\delta_2)(\frac13\delta_1-\delta_2-\frac13)(\frac13\delta_1-\delta_2-\frac23)
\end{align*}
The singular locus contains inside
\[
D_A \cup  D_1 \cup D_{v_1^s}\cup D_{v_2^s}\cup D_{v_3^s}\cup D_{v_4^s}
\]
where $D_A = (1+z_1)^3+z_1^3z_2$, $D_1 = 1+z_2$. In this affine chart, we have $D_{v_1^s}$ intersects $D_{v_2^s}$ transversely at origin; 
The $D_A$ intersects $D_{v_2^s} = \bV(z_2)$ at $(z_1,z_2) = (-1,0)$ with 3-tangency and thus requires to be blown up three times, with the exceptional divisors denoted as $E_1,E_2,E_3$;
The $D_1$ intersects $D_{v_1^s}$ transversally at one point $(z_1,z_2) = (0,-1)$. 
Now, there are intersections outside this affine chart. Now, switch back to the original $z_1,z_2$ coordinates before the rescaling and let
\[
z_1' = \frac{Z_1Z_2^{\frac13}}{Z_4},\quad z_2' = \frac{Z_3}{Z_2^{\frac13}},
\]
we have
\[
z_1 = z_1'z_2',\quad z_2 = z_2'^{-3},
\]
Since the coordinates change is étale away from the boundaries, the Picard-Fuchs ideal becomes
\[
\begin{split}
 & P_1 = \delta_1'\delta_2' +3z_1'z_2'(3\delta_1'+1)(3\delta_1'+2)\\
 & P_2 = z_2'^3(\delta_1'-\delta_2')^3-\delta_2'(\delta_2'-1)(\delta_2'-2)
 \end{split}
\]
Similar, by a rescale $z_{1}'':=3^4\cdot z_{1}'$, $z_{2}'':=3^{-1}\cdot z_{2}'$, and rename $z_{1,2}''$ as the new $z_{1,2}'$, then $D_A = (1+z_1'z_2')^3+z_1'^3$, $D_1 = 1+z_2^{-3}$, and
\[
\begin{split}
 & P_1 = \delta_1'\delta_2' +z_1'z_2'(\delta_1'+\frac13)(\delta_1'+\frac23)\\
 & P_2 = 27z_2'^3(\delta_1'-\delta_2')^3-\delta_2'(\delta_2'-1)(\delta_2'-2),
 \end{split}
\]
In this affine chart we have additionally two cones that correspond to $D_{v_2^s}$ intersects $D_{v_3^s}$ transversely at origin, and $D_A$ intersects $D_{v_3^s}$ transversely at $(-1,0)$. The last affine chart that we would want to consider is 
\[
z_1'' = z_1^{-1},\quad z_2'' = z_2,
\]
where $z_i$ are the original $z_i$ before rescale. Then the Picard-Fuchs ideal is
\begin{align*}
    & P_1 = z_1\delta_1(\delta_1+3\delta_2) +3(3\delta_1-1)(3\delta_1-2)\\
    & P_2 = \delta_2^3+z_2(\delta_1+3\delta_2)(\delta_1+3\delta_2+1)(\delta_1+3\delta_2+2),
\end{align*}
Similar, by a rescale $z_{1}''':=3^{-3}\cdot z_{1}'$, $z_{2}'':=3^{3}\cdot z_{2}'$, and rename $z_{1,2}''$ as the new $z_{1,2}'$, then $D_A = (1+z_1)^3+z_2$, $D_1 = 1+z_2$, and
\begin{align*}
    & P_1 = z_1\delta_1(\delta_1+3\delta_2) +(\delta_1-\frac13)(\delta_1-\frac23)\\
    & P_2 = \delta_2^3+z_2(\frac13\delta_1+\delta_2)(\frac13\delta_1+\delta_2+\frac13)(\frac13\delta_1+\delta_2+\frac23),
\end{align*}
This time, $D_A$ intersects $D_1$ and $D_{v_4^s}$ at a triple intersection point $(0,-1)$. Thus we need to blow up this point once, and the exceptional divisor is denoted as $E_0$. Since $D_{v_4^s}$ has finite monodromy, the only new possible two-dimension cones in this chart are those associated to $D_A\cap E_0$ and $D_1\cap E_0$.

\begin{remark}
    All the rescales in this section can be done simultaneously in the homogeneous coordinates $Z_i$ by rescaling $Z_1,Z_2$ by a $3^3$ factor.
\end{remark}

\subsubsection{Nilpotent cone and types of LMHS}
 Similarly, we can take the global multi-valued frame 
\[
w = (\Omega,\delta_1\Omega,\delta_2\Omega,\delta_1\delta_2\Omega,\delta_2^2\Omega,\delta_2^2\delta_1\Omega),
\]
and under this basis, the nilpotent cone at $z_1=z_2=0$ are represented as
\[
\overline N_1 =\left[
\begin{array}{cccccc}\label{eqn:monomatricesv238}
    0 & 0  & 0 & 0 & 0 & 0  \\ 
    1 & 0  & 0 & 0 & 0 & 0 \\ 
    0 & 0 &  0&  0&  0&  0\\ 
    0 & 3 & 1 & 0 & 0 & 0  \\ 
    0 & 0 & 0 & 0 & 0 & 0  \\ 
    0 & 0 & 0 & 3 & 1 &  0 \\ 
\end{array}
\right],\quad \overline N_2 =\left[
\begin{array}{cccccc}
    0 & 0  & 0 & 0 & 0 & 0  \\ 
    0 & 0  & 0 & 0 & 0 & 0 \\ 
    1 & 0 &  0&  0&  0&  0\\ 
    0 & 1 & 0 & 0 & 0 & 0  \\ 
    0 & 0 & 1 & 0 & 0 & 0  \\ 
    0 & 0 & 0 & 1 & 0 &  0 \\ 
\end{array}
\right].
\]
The possible 2-dimensional cones with the types of LMHS are given as follows
\begin{align}
 \label{cone21}    \mathrm{Cone}(D_{v_1^s},D_{v_2^s}) & = \braket{IV_1|IV_2|III_0}\\
 \label{cone22}    \mathrm{Cone}(D_{v_2^s},E_3) & = \braket{III_0|III_0(IV_2)|III_0}\\
 \label{cone23}    \mathrm{Cone}(E_1,E_2) & = \braket{III_0|III_0(IV_2)|III_0}\\
 \label{cone24}    \mathrm{Cone}(E_2,E_3) & = \braket{III_0|III_0(IV_2)|III_0}\\
 \label{cone25}    \mathrm{Cone}(D_A,E_3) & = \braket{I_1|III_0(IV_2)|III_0}\\
    \label{cone26}    \mathrm{Cone}(D_{v_1^s},D_1) & = \braket{IV_1|IV_2|I_1}\\
    \label{cone27}    \mathrm{Cone}(D_{v_2^s},D_{v_3^s}) & = \braket{III_0|III_0(IV_2)|I_1}\\
    \label{cone28}    \mathrm{Cone}(D_{v_3^s},D_A) & = \braket{I_1|I_2(I_1)|I_1}\\
    \label{cone29}    \mathrm{Cone}(D_{v_1^s},D_{v_3^s}) & = \braket{IV_1|IV_2|I_1}\\
        \label{cone210}    \mathrm{Cone}(D_{A},E_0) & = \braket{I_1|III_0(IV_2)|III_0}\\
        \label{cone211}    \mathrm{Cone}(D_{1},E_0) & = \braket{I_1|III_0(IV_2)|III_0}
\end{align}

\subsection{Symplectic form}
Similar to the calculation in Section \ref{V229:sec.Yuka}, one can calculate the symplectic form by calculating the Yukawa couplings. In particular, form $\int_V\Omega\wedge \delta_i P_1\Omega$ and $\int_V\Omega\wedge P_2\Omega$, we have
\begin{align*}
    & (1+z_1)K^{3,0} -3K^{2,1} = 0,\\
    &(1+z_1)K^{2,1} -3K^{1,2} =0,\\
    & z_2K^{3,0}- 9z_2K^{2,1} +27z_2K^{1,2} - 27(1+z_2)K^{0,3} =0,
\end{align*}
where the last equation gives
\[
K^{0,3} = \frac{z_2(1+3z_1+3z_1^2)}{27(1+z_2)}K^{3,0}.
\]

Substitute these into $\int_{V}\Omega\wedge \delta_1^2P_1\Omega=0$, $\int_{V}\Omega\wedge \delta_2P_2\Omega=0$, we have
\begin{align*}
& \delta_{z_1}K^{3,0} = \frac{z_1(6z_1+3z_1^2z_2+3z_1^2+3)}{D_A} K^{3,0},\\
& \delta_{z_2}K^{3,0} = \frac{z_1^3z_2}{D_A} K^{3,0}.\\
\end{align*}
Therefore, we have $K^{3,0} = c\cdot D_A$, where $c$ is some constant. Similar to the calculation in \eqref{QinYuk}, the symplectic form is given as
\[
Q =\frac{c}{9}\cdot \left[
\begin{array}{cccccc}
    0 & 0  & 0 & 0 & 0 & 1  \\ 
    0 & 0  & 0 & -1 & 0 & 0 \\ 
    0 & 0 &  0&  -3&  -1&  0\\ 
    0 & 1 & 3 & 0 & 0 & 0  \\ 
    0 & 0 & 1 & 0 & 0 & 0  \\ 
    -1 & 0 & 0 & 0 & 0 &  0 \\ 
\end{array}
\right].
\]

\subsection{$V_{(2,86)}$, the mirror octic}
This is the smooth Calabi-Yau 3-fold mirror to the 
the smooth anti-canonical hypersurface in $\bP[1,1,2,2,2]$. This family is very similar to the case of $V_{(2,38)}$. In particular, the simplified moduli space is the same as the complex moduli space. The mirror symmetry of this family is studied in detail in \cite{COFK94,CK99}. 

The first author calculated the nilpotent fan for the family in \cite{Chen23}. The two-dimensional nilpotent cones are 
\begin{align*}\label{eqn:conesv286}
    & \sigma_{12} = 
    \braket{IV_2|IV_2|II_1},
    && \sigma_{13} = \braket{IV_2|IV_2|I_1},\\
    &\sigma_{34} = 
    \braket{I_1|I_2|I_1},
    && \sigma_{45} =
    \braket{I_1|II_1|II_0}, \\
    & \sigma_{52} =
    \braket{II_0|II_1|II_1},
    && \sigma_{63} = 
    \braket{I_1|I_2|I_1}, \\
    & \sigma_{67} =
    \braket{I_1|I_2|I_1}.
\end{align*}
Under the following basis of $F^0_{\mathrm{lim}}$
\[
w = (\Omega,\delta_1\Omega,\delta_2\Omega,\delta_1\delta_2\Omega,\delta_1^2\Omega,\delta_1^2\delta_2\Omega),
\]
the nilpotent operators $\overline N_1,\overline N_2$ that generates the cone $\sigma_{12}$ are represented in this basis as
\begin{equation}\label{eqn:monomatricesv286}
\overline N_1 =\left[
\begin{array}{cccccc}
    0 & 0  & 0 & 0 & 0 & 0  \\ 
    1 & 0  & 0 & 0 & 0 & 0 \\ 
    0 & 0 &  0&  0&  0&  0\\ 
    0 & 0 & 1 & 0 & 0 & 0  \\ 
    0 & 1 & 0 & 0 & 0 & 0  \\ 
    0 & 0 & 0 & 1 & 2 &  0 \\ 
\end{array}
\right],\quad \overline N_2 =\left[
\begin{array}{cccccc}
    0 & 0  & 0 & 0 & 0 & 0  \\ 
    0 & 0  & 0 & 0 & 0 & 0 \\ 
    1 & 0 &  0&  0&  0&  0\\ 
    0 & 1 & 0 & 0 & 0 & 0  \\ 
    0 & 0 & 0 & 0 & 0 & 0  \\ 
    0 & 0 & 0 & 0 & 1 &  0 \\ 
\end{array}
\right],
\end{equation}
and the symplectic form is
\begin{equation}
Q = c\cdot\left[\begin{array}{cccccc}0 & 0 & 0 & 0 & 0 & 1 \\0 & 0 & 0 & -1 & -2 & 0 \\0 & 0 & 0 & 0 & -1 & 0 \\0 & 1 & 0 & 0 & 0 & 0 \\0 & 2 & 1 & 0 & 0 & 0 \\-1 & 0 & 0 & 0 & 0 & 0\end{array}\right],
\end{equation}
where $c$ is some constant.


\section{Computing the generic degree}\label{sec05}

In this section we show the main Theorem \ref{Thm:mainthmcomputedegree} by explicitly computing $\mathrm{deg}(\overline{\Phi_s})$ (see Proposition \ref{prop:localtoglobaldegree}) around a specific boundary point $s\in \olS-S$. This will fullfill the proof of Theorem \ref{Thm:mainthmcomputedegree}.
\subsection{$V_{(2,29)}$}
The period map is $\Phi: \cM_{\mathrm{simp}}=:S\rightarrow \Gamma\backslash D$. We choose a minimal resolution $\overline{\cM_{\mathrm{simp}}}$ of $\bP^2$ such that $\overline{\cM_{\mathrm{simp}}}-\cM_{\mathrm{simp}}$ is a normal crossing divisor. 

Take $s=[1:0:0]\in \bP^2$. Since locally around $\bP^2-\cM_{\mathrm{simp}}$ is normal-crossing, we may think $s$ as a codimension $2$ boundary point in $\overline{\cM_{\mathrm{simp}}}-\cM_{\mathrm{simp}}$.

Types listed in \eqref{V229:cone1} - \eqref{V229:cone5} imply $(\sigma_s, F_s)$ is the only local nilpotent orbit with type $\langle\mathrm{III}_0|\mathrm{IV}_2|\mathrm{III}_0 \rangle$, and $\mathrm{deg}(\overline{\varphi_s}|_{\overline{\wp_s}})=2$ as a consequence of Sections \ref{SubSec:4.1.1} and $\ref{SubSec:4.1.2}$. Therefore to show the Theorem \ref{Thm:mainthmcomputedegree} for $V_{(2,29)}$, we only have to show $\mathrm{deg}(\overline{\Phi_s})=1$.

The matrices of the monodromy cone $\sigma_s=\langle N_1, N_2\rangle$ are computed in \ref{eqn:monomatricesv229}. In a local lifting $\tilde{\Phi}_U$ around $s\in U\simeq \Delta^2$, denote $l(z_i):=\frac{\log(z_i)}{2\pi i}$, the local lifted period map \eqref{eqn:localliftperiodmap} takes the form:
\begin{equation}
    \tilde{\Phi}_U(z_1,z_2)=\exp(l(z_1)N_1+l(z_2)N_2)\psi_U(z_1,z_2), \ (z_1,z_2)\in \Delta^2.
\end{equation}
Let $F_s:=\psi_U(0,0)\in \check{D}$, under the basis \eqref{eqn:sympbasisV229} which we denote as $(e_1,...,e_6)$, we make suppose 
\begin{equation}
    F^3\psi_U(z_1,z_2)=\langle(1,\zeta_1(z_1,z_2),...,\zeta_5(z_1,z_2))\rangle.
\end{equation}
By computation we have
\begin{equation}\label{eqn:periodmatrixcomputationV229}
    F^3\tilde{\Phi}_U(z_1,z_2)=\langle(1,\zeta_1(z_1,z_2)+l(z_1),\zeta_2(z_1,z_2)+l(z_2),*,*,*)\rangle
\end{equation}
where *'s are irrelevant terms. This implies we may construct the map $U\rightarrow \Delta^2$ by 
\begin{eqnarray}\label{eqn:localnbhisomorphismV229}
    &(z_1,z_2)\rightarrow (\exp(2\pi i\langle F^3\tilde{\Phi}_U(z_1,z_2), e_5\rangle), \exp(2\pi i\langle F^3\tilde{\Phi}_U(z_1,z_2), e_4\rangle))\\
    &=(z_1\exp(2\pi i\zeta_1(z_1,z_2)), z_2\exp(2\pi i\zeta_2(z_1,z_2))).
\end{eqnarray}
Since the Jacobian matrix of this map at $(0,0)$ is invertible, $(0,0)$ is not a branched point and \eqref{eqn:localnbhisomorphismV229} is a local isomorphism. This implies the desired $\mathrm{deg}(\overline{\Phi_s})=1$.

\subsection{$V_{(2,38)}$}

We use the boundary point \eqref{cone21} and its corresponding monodromy matrices \eqref{eqn:monomatricesv238}. Using the same notations as the last subsection, the map \eqref{eqn:periodmatrixcomputationV229} now read as:
\begin{equation}\label{eqn:periodmatrixcomputationV238}
    F^3\tilde{\Phi}_U(z_1,z_2)=\langle(1,\zeta_1(z_1,z_2)+l(z_2),\zeta_2(z_1,z_2)+l(z_1),*,*,*)\rangle,
\end{equation}
and the analog of map \eqref{eqn:localnbhisomorphismV229} is read as:
\begin{equation}\label{eqn:localnbhisomorphismV238}
    (z_1,z_2)\rightarrow (z_2\exp(2\pi i\zeta_1(z_1,z_2)), z_1\exp(2\pi i\zeta_2(z_1,z_2))).
\end{equation}
The same argument shows $\mathrm{deg}(\overline{\Phi_s})=1$. Note that among monodromy cones \eqref{cone21} - \eqref{cone211}, the LMHS type of \eqref{cone21} is unique, and \eqref{cone21} does not have any non-trivial torsion automorphisms other than rescaling because of Proposition \ref{prop:localtoglobalKNUmapdegree}, thus $\mathrm{deg}(\overline{\varphi_s}|_{\overline{\wp_s}})=1$. 

\subsection{$V_{(2,86)}$}

We use the boundary point corresponds to the cone $\sigma_{12}$ whose monodromy matrices are \eqref{eqn:monomatricesv286}. Again using the same notations as the last subsection, the map \eqref{eqn:periodmatrixcomputationV229} now read as:
\begin{equation}\label{eqn:periodmatrixcomputationV286}
    F^3\tilde{\Phi}_U(z_1,z_2)=\langle(1,\zeta_1(z_1,z_2)+l(z_1),\zeta_2(z_1,z_2)+l(z_2),*,*,*)\rangle,
\end{equation}
and the analog of map \eqref{eqn:localnbhisomorphismV229} is read as:
\begin{equation}\label{eqn:localnbhisomorphismV286}
    (z_1,z_2)\rightarrow (z_1\exp(2\pi i\zeta_1(z_1,z_2)), z_2\exp(2\pi i\zeta_2(z_1,z_2))).
\end{equation}
The same argument shows $\mathrm{deg}(\overline{\Phi_s})=1$. The same reasons as those for $V_{(2,38)}$ show that $\mathrm{deg}(\overline{\varphi_s}|_{\overline{\wp_s}})=1$. The proof of Theorem \ref{Thm:mainthmcomputedegree} is now completed.

\printbibliography

@article{AGM93,
  title={The monomial-divisor mirror map},
  author={Aspinwall, Paul S and Greene, Brian R and Morrison, David R},
  journal={arXiv preprint alg-geom/9309007},
  year={1993}
}

@article{AGM94,
  title={Calabi-Yau moduli space, mirror manifolds and spacetime topology change in string theory},
  author={Aspinwall, Paul S and Greene, Brian R and Morrison, David R},
  journal={Nuclear Physics B},
  volume={416},
  number={2},
  pages={414--480},
  year={1994},
  publisher={Elsevier}
}

@article{Bat94,
  title={Dual polyhedra and mirror symmetry for Calabi--Yau hypersurfaces in toric varieties},
  author={Batyrev, V},
  journal={J. Algebraic Geom.},
  volume={3},
  pages={493},
  year={1994}
}

@article{BC94,
  title={On the Hodge structure of projective hypersurfaces in toric varieties},
  author={Batyrev, Victor V and Cox, David A},
  journal={Duke Math. J.},
  volume={76},
  number={1},
  pages={293--338},
  year={1994}
}

@article{BBT22,
author = {Benjamin Bakker and Yohan Brunebarbe and Jacob Tsimerman},
title = {o-minimal GAGA and a conjecture of Griffiths},
journal = {Inventiones mathematicae},
volume = {232},
year = {2022},
pages = {163-228}
}

@article{COFK94,
  title={Mirror symmetry for two-parameter models (I)},
  author={Candelas, Philip and De La Ossa, Xenia and Font, Anamaria and Katz, Sheldon and Morrison, David R},
  journal={Nuclear Physics B},
  volume={416},
  number={2},
  pages={481--538},
  year={1994},
  publisher={Elsevier}
}

@article{Chen23,
  title={Completion of a period map of hodge type (1, 2, 2, 1)},
  author={Chen, Chongyao},
  journal={arXiv preprint arXiv:2311.10212},
  year={2023}
}

@article{CK82,
author = {Eduardo Cattani and Aroldo Kaplan},
title = {Polarized mixed Hodge structures and the local monodromy of a variation of Hodge structure},
journal = {Inventiones mathematicae},
volume = {67},
year = {1982},
pages = {101–115}
}

@article{CKS86,
author = {Eduardo Cattani and Aroldo Kaplan and Wilfried Schmid},
title = {Degeneration of Hodge Structures},
journal = {Annals of Mathematics},
volume = {123},
number = {03},
year = {1986},
pages = {457-535}
}

@book{CK99,
  title={Mirror symmetry and algebraic geometry},
  author={Cox, David A and Katz, Sheldon},
  volume={68},
  year={1999},
  publisher={American Mathematical Society Providence, RI}
}

@book{CLS11,
author = {David A. Cox and John B. Little and Henry K. Schenck},
title = {Toric Varieties
},
publisher = {American Mathematical Society},
series = {Graduate Studies in Mathematics},
volume = {124},
year = {2011}
}

@misc{DR23,
      title={Completion of two-parameter period maps by nilpotent orbits}, 
      author={Haohua Deng and Colleen Robles},
      year={2023},
      eprint={2312.00542},
      archivePrefix={arXiv},
      primaryClass={math.AG}
}

@article{Fil23,
  title={Global properties of some weight 3 variations of Hodge structure},
  author={Simion Filip},
  journal={Proceedings of the 8th European Congress of Mathematics},
  year={2023},
  page = {553-568}
}

@book{GKZ,
  title={A-discriminants},
  author={Gelfand, Israel M and Kapranov, Mikhail M and Zelevinsky, Andrei V and Gelfand, Israel M and Kapranov, Mikhail M and Zelevinsky, Andrei V},
  year={1994},
  publisher={Springer}
}

@article{HK21,
  title={Degenerating Hodge structure of one–parameter family of Calabi–Yau threefolds},
  author={Tatsuki Hayama and Atsushi Kanazawa},
  journal={Asian Journal of Mathematics},
  volume={25},
  number = {1},
  pages={31--42},
  year={2021},
}

@article{Hot91,
  title={Equivariant D-modules},
  author={Hotta, Ryoshi},
  journal={arXiv preprint math/9805021},
  year={1991},
  publisher={Citeseer}
}

@article{HKTY95,
  title={Mirror symmetry, mirror map and applications to Calabi-Yau hypersurfaces},
  author={Hosono, Shinobu and Klemm, Albrecht and Thiesen, S and Yau, Shing-Tung},
  journal={Communications in Mathematical Physics},
  volume={167},
  pages={301--350},
  year={1995},
  publisher={Springer}
}

@article{HLY16,
  title={Period integrals and the Riemann--Hilbert correspondence},
  author={Huang, An and Lian, Bong H and Zhu, Xinwen},
  journal={Journal of differential geometry},
  volume={104},
  number={2},
  pages={325--369},
  year={2016},
  publisher={Lehigh University}
}

@book{HTT07,
  title={D-modules, perverse sheaves, and representation theory},
  author={Hotta, Ryoshi and Tanisaki, Toshiyuki},
  volume={236},
  year={2007},
  publisher={Springer Science \& Business Media}
}

@article{KNU10,
author = {Kazuya Kato and Chikara Nakayama and Sampei Usui},
title = {Classifying spaces of degenerating mixed Hodge structures, III: Spaces of nilpotent orbits},
journal = {J. Algebraic Geom.},
volume = {22},
year = {2013},
pages = {671-772}
}

@article{KSZ91,
  title={Quotients of toric varieties},
  author={Kapranov, Mikhail M and Sturmfels, Bernd and Zelevinsky, Andrei V},
  journal={Mathematische Annalen},
  volume={290},
  pages={643--655},
  year={1991},
  publisher={Springer}
}

@article{KPR19,
author = {Kerr, M and Pearlstein, G and Robles, C},
title = {Polarized relations on horizontal SL(2)'s},
journal = {Documenta Mathematica},
volume = {24},
year = {2019},
pages = {1295-1360}
}

@book{KU08,
author = {Kazuya Kato and Sampei Usui},
title = {Classifying Spaces of Degenerating Polarized Hodge Structures},
series = {Annals of Mathematics Studies},
volume = {169},
publisher = {Princeton University Press},
year = {2008}
}

@article{LSY13,
  title={Periodic integrals and tautological systems},
  author={Lian, Bong H and Song, Ruifang and Yau, Shing-Tung},
  journal={Journal of the European Mathematical Society},
  volume={15},
  number={4},
  pages={1457--1483},
  year={2013}
}

@article{Sch73,
author = {Wilfried Schmid},
title = {Variation of Hodge Structure: The Singularities of the Period Mapping},
journal = {Inventiones mathematicae},
volume = {22},
year = {1973},
pages = {211-320}
}

@article{Shi09,
author = {Kennichiro Shirakawa},
title = {Generic Torelli theorem for one-parameter mirror families to weighted hypersurfaces},
journal = {Proc. Japan Acad. Ser. A Math. Sci.},
volume = {85},
number = {10},
year = {2009},
pages = {167-170}
}

@article{Usu06,
author = {Sampei Usui},
title = {Images of extended period maps},
journal = {J. Algebraic Geom.},
volume = {15},
year = {2006},
pages = {603-621}
}

@article{Usu08,
author = {Sampei Usui},
title = {Generic Torelli theorem for quintic-mirror family},
journal = {Proc. Japan Acad. Ser. A Math. Sci.},
volume = {84},
number = {8},
year = {2008},
pages = {143-146}
}
\end{document}